\documentclass[11pt]{amsart}
\usepackage{amsmath}
\usepackage{latexsym,amsfonts,amssymb,mathrsfs}
\usepackage{mathdots}
\usepackage{color}
\usepackage[all,2cell,color]{xy}
\usepackage{enumitem} 
\usepackage{bbm}
\usepackage{tikz}
\usepackage{tikz-cd}
\usepackage[T1]{fontenc}
 \usepackage[utf8]{inputenc}
 \usepackage{blkarray} 
 \usepackage{mathtools} 
 \usepackage{longtable}
 \usepackage{multicol} 
 \usepackage{hyperref}
\usepackage[capitalise]{cleveref}

\UseAllTwocells
\usepackage[
textwidth=3cm,
textsize=small,
colorinlistoftodos]
{todonotes}

\usetikzlibrary{positioning} 
\usetikzlibrary{decorations.pathreplacing}
\usetikzlibrary{arrows, decorations.markings} 
\pgfarrowsdeclarecombine{twotriang}{twotriang}{stealth}{stealth}%
{stealth}{stealth}
\tikzset{arrow/.style={-stealth}}
\tikzset{arrowshorter/.style={-stealth, shorten <=2pt, shorten >=2pt}}
\tikzset{arrowmuchshorter/.style={-stealth, shorten <=7pt, shorten >=6pt}}
\tikzset{mono/.style={>-stealth}} 
\tikzset{epi/.style={-twotriang}} 
\tikzset{twoarrowlonger/.style={double,double distance=1.5pt,
shorten <=5pt,shorten >=6pt,
decoration={markings,mark=at position -4pt with {\arrow[scale=1.75]{>}}},
preaction={decorate}}} 

\tikzset{twoarrow/.style={double,double distance=1.5pt,
shorten <=6pt,shorten >=7pt, 
decoration={markings,mark=at position -4pt
with {\arrow[scale=1.75]{>}}},
preaction={decorate} 
}
}
\tikzset{%
    symbol/.style={%
        draw=none,
        every to/.append style={%
            edge node={node [sloped, allow upside down, auto=false]{$#1$}}}
    }
}

\tikzset{mapstikz/.style={-stealth, 
decoration={markings,mark=at position 0pt with {\arrow[scale=0.5]{|}}}, preaction={decorate}}}
\usetikzlibrary{backgrounds,shapes,arrows,calc,patterns} 

\pgfdeclarelayer{background}
\pgfsetlayers{background,main}

\theoremstyle{plain}   
\newtheorem{thm}{Theorem}[section] 
\makeatletter\let\c@thm\c@thm\makeatother
\newtheorem{cor}{Corollary}[section]
\makeatletter\let\c@cor\c@thm\makeatother
\newtheorem{lem}{Lemma}[section]
\makeatletter\let\c@lem\c@thm\makeatother
\newtheorem{prop}{Proposition}[section]
\makeatletter\let\c@prop\c@thm\makeatother

\makeatletter\let\c@claim\c@thm\makeatother

\makeatletter\let\c@conjecture\c@thm\makeatother

\newtheorem{thmalph}{Theorem}

\newtheorem{coralph}{Corollary}

\theoremstyle{definition}

\newtheorem{defn}{Definition}[section]
\makeatletter\let\c@defn\c@thm\makeatother

\makeatletter\let\c@const\c@thm\makeatother
\newtheorem{notn}{Notation}[section]
\makeatletter\let\c@notn\c@thm\makeatother

\makeatletter\let\c@convention\c@thm\makeatother

\theoremstyle{remark}

\newtheorem{rmk}{Remark}[section]
\makeatletter\let\c@rmk\c@thm\makeatother
\newtheorem{ex}{Example}[section]
\makeatletter\let\c@ex\c@thm\makeatother

\makeatletter\let\c@observation\c@thm\makeatother

\makeatletter\let\c@warning\c@thm\makeatother

\makeatletter\let\c@digression\c@thm\makeatother

\makeatletter\let\c@answ\c@thm\makeatother

\makeatletter\let\c@answ\c@thm\makeatother

\makeatletter\let\c@aside\c@thm\makeatother

\makeatletter
\let\c@equation\c@thm
\numberwithin{equation}{section}
\makeatother

\newcommand{\newrefformat}[2]{}

\crefname{lem}{Lemma}{Lemmas}
\crefname{thm}{Theorem}{Theorems}
\crefname{defn}{Definition}{Definitions}
\crefname{notn}{Notation}{Notations}
\crefname{const}{Construction}{Constructions}
\crefname{prop}{Proposition}{Propositions}
\crefname{rmk}{Remark}{Remarks}
\crefname{cor}{Corollary}{Corollaries}
\crefname{equation}{Display}{Displays}
\crefname{ex}{Example}{Examples}
\crefname{thmalph}{Theorem}{Theorems}
\crefname{answ}{Answer}{Answers}
\crefname{crzcond}{Crazy Condition}{Crazy Conditions}

\newcommand{\cA}{\mathcal{A}}

\newcommand{\cC}{\mathcal{C}}
\newcommand{\cD}{\mathcal{D}}

\newcommand{\cM}{\mathcal{M}}

\newcommand{\cO}{\mathcal{O}}
\newcommand{\cP}{\mathcal{P}}
\newcommand{\cQ}{\mathcal{Q}}

\newcommand{\cS}{\mathcal{S}}

\newcommand{\cat}{\cC\!\mathit{at}}

\newcommand{\set}{\cS\!\mathit{et}}

\newcommand{\sset}{\mathit{s}\set}

\DeclareMathOperator{\Map}{Map}

\newcommand{\inttrunc}[1]{\tau_{\leq n}^{\text{i}}}

 \newcommand{\NRS}{N^{\operatorname{RS}}}

        \newcommand{\Nnat}{N^{\natural}}

   \newcommand{\matn}[2]{\operatorname{Mat}_{#1}{#2}}
      \newcommand{\mat}[1]{\operatorname{Mat}{#1}}

\DeclareMathOperator{\colim}{colim}

\DeclareMathOperator{\id}{id}

\newcommand{\aamalg}[1]{\underset{#1}{{\amalg}}} 
\DeclareMathOperator{\op}{op}

\newcommand{\eqDelta}{\Delta[3]_{eq}}

\newcommand\simpfver[9]{%
    \def\tempa{#1}%
    \def\tempb{#2}%
    \def\tempc{#3}%
    \def\tempd{#4}%
    \def\tempe{#5}%
    \def\tempf{#6}%
    \def\tempg{#7}%
    \def\temph{#8}%
    \def\tempi{#9}%
}
\newcommand\simpfvercontinued[5]{%
    \begin{tikzcd}[column sep=1.3cm, row sep=0.4cm, baseline=(current  bounding  box.center), ampersand replacement=\&]
 #5 \& #4 \arrow[l, "\tempc" swap] \&[-7mm] \&[-7mm] #5 \& #4 \arrow[l, "\tempc" swap]\\
 {} \& \&[-7mm] {=} \&[-7mm] \& \\
 #2 \arrow[r, "\tempa" swap]\arrow[uu, "\tempd", ""{name=a031,inner sep=2pt, swap} ] \arrow[uur, "\tempe"{near start, xshift=0.15cm}, ""{name=a02,inner sep=2pt, swap}]\& #3 \arrow[uu, "\tempb" swap]\&[-7mm] \&[-7mm] #2 \arrow[uu, "\tempd", ""{name=a032,inner sep=2pt, swap}]\arrow[r, "\tempa" swap]\& #3\arrow[uu, "\tempb" swap] \arrow[uul, "\tempf"{near end, xshift=0.25cm}, ""{name=a13,inner sep=2pt, swap}]
 \arrow[Rightarrow, from=a031, to=1-2, shorten >= 0.3cm, "\tempg" ]
 \arrow[Rightarrow, from=a02, to=3-2, shorten >= 0.3cm, "\temph" swap ]
 \arrow[Rightarrow, from=a032, to=3-5, shorten >= 0.3cm, "\tempi" swap ]
 \arrow[Rightarrow, from=a13, to=1-5, shorten >= 0.3cm, "#1" ]
 \end{tikzcd}
}

\tikzset{%
scalearrow/.style n args={3}{
  decoration={
    markings,
    mark=at position (1-#1)/2*\pgfdecoratedpathlength
      with {\coordinate (#2);},
    mark=at position (1+#1)/2*\pgfdecoratedpathlength
      with {\coordinate (#3);},
    },
  postaction=decorate,
  }
}

\tikzset{mono/.style={>-stealth}} 
\tikzset{epi/.style={-twotriang}} 

\makeatletter
\@namedef{subjclassname@2020}{\textup{2020} Mathematics Subject Classification}
\makeatother

\author{Viktoriya Ozornova}
\address{Fakult\"at f\"ur Mathematik, Ruhr-Universit\"at Bochum, Bochum, Germany}
\email{viktoriya.ozornova@rub.de}

\author{Martina Rovelli}
\address{Mathematical Sciences Institute,
The Australian National University,
Canberra, Australia
}
\email{martina.rovelli@anu.edu.au}


\begin{document}

\title{Fundamental pushouts of $n$-complicial sets}

\maketitle

\begin{abstract}
The paper focuses on investigating how certain relations between strict $n$-categories are preserved in a particular implementation of $(\infty,n)$-categories, given by saturated $n$-complicial sets. In this model, we show that the $(\infty,n)$-categorical nerve of $n$-categories is homotopically compatible with $1$-categorical suspension and wedge. As an application, we show that certain pushouts encoding composition in $n$-categories are homotopy pushouts of saturated $n$-complicial sets.
    \end{abstract}

  \section*{Introduction}
  Since 1950s, category theory has established itself as a language to phrase mathematical phenomena in a uniform way. 
  Recent developments in the study of the cobordism hypothesis, in derived algebraic geometry and in brave new algebra, highlighted the presence and role played by higher morphisms, as well as the fact that axioms defining a categorical structure should be weakened, replacing equalities with higher isomorphisms. This perspective sparked new interest in the study of generalizations of the notion of an ordinary category, in the form of an $n$-category and then of an $(\infty,n)$-category.
  
    While it is still unfolding its significance in algebraic topology, higher category theory arose in 1960s with the original purpose of encoding non-abelian cohomology into the language of \emph{$n$-categories}.
     The notion of an $n$-category encapsulates the idea that beyond objects and morphisms between them, there are also morphisms between morphisms, called $2$-morphisms, morphisms between those, called $3$-morphisms, and so on up to level $n$. All these morphisms compose associatively along morphisms of lower dimensions.
  
  Composition of morphisms was traditionally requested to satisfy strict equational conditions, such as strict associativity, and this led to a rich theory of strict enriched category theory.
However, many examples of interest that naturally present a higher categorical structure, such as several categories of cobordisms, derived categories, or the categorical structure given by points, paths and higher homotopies in a topological space, fail to satisfy these axioms.
  
  Seemingly very different in nature, the notion of an $n$-category had to then be weakened in order to accommodate homotopical phenomena, becoming itself a homotopical notion, and in the 1990's the notion of an \emph{$(\infty,n)$-category} started making its way.
  An $(\infty,n)$-category should consist of objects, regarded as $0$-morphisms, and $k$-morphisms between $(k-1)$-morphisms for any $k$; these morphisms must moreover compose weakly associatively along morphisms of a lower dimension and  are all weakly invertible for $k>n$.
  While the theory of strict $n$-categories is unambiguous, the defining guidelines for the notion of $(\infty,n)$-category have been given a precise meaning in different models. All models are conjecturally equivalent, although some comparisons showing equivalences of the corresponding homotopy theories are still missing.
  
  Regardless of the model, the collection of $(\infty,n)$-categories should assemble at least into an $(\infty,1)$-category $(\infty,n)\cat$, enlarging the $(\infty,1)$-category of strict $n$-categories $n\cat$, and the inclusion of $(\infty,1)$-categories
   $N\colon n\cat\hookrightarrow(\infty,n)\cat$ has been realized in many models, often implemented by a type of nerve construction. It is interesting to understand how this embedding behaves with natural constructions of a categorical flavour, given that nerve constructions typically behave poorly with respect to constructions involving left adjoint functors and colimits.

The goals of this article is to show that in a specific model of $(\infty,n)$-categories, Verity's \emph{saturated $n$-complicial sets} \cite{VerityComplicialAMS,EmilyNotes,or},  two types of constructions, \emph{suspension} and \emph{wedge}, are compatible with the embedding.
As a motivating application, we show that the nerve embedding preserves certain fundamental relations between $n$-categories, that encode composition and invertibility of morphisms.

In other models, such as Barwick's \emph{$n$-fold complete Segal spaces} \cite{BarwickThesis, LurieGoodwillie}, Rezk's \emph{$\Theta_n$-spaces} \cite{rezkTheta} and Ara's \emph{$n$-quasicategories} \cite{Ara}, the analogous statements are essentially part of the axioms. However, given the lack of model comparisons with saturated $n$-complicial sets for $n\ge3$ and the unexplored compatibility of existing model comparisons with the nerve embedding for $n=2$, the result could not be imported at no cost.

  In \cref{sec:SuspensionResults} we introduce the suspension of $1$- and $(\infty,1)$-categories, which can be seen as a left adjoint to taking the hom $1$- or $(\infty,1)$-category between two objects of a $2$- or $(\infty,2)$-category.
  Roughly speaking, the \emph{suspension} of a $1$- or $(\infty,1)$-category $\cP$ is a $2$- or $(\infty,2)$-category with two objects and a unique interesting hom-category given by $\cP$.
  Then, we show in \cref{sec:SuspensionProofs} as \cref{ThmB} the following compatibility of nerve and suspension.

  \begin{thmalph}
  \label{ThmBIntro}
  In the model of saturated $2$-complicial sets, for any $1$-category $\cP$ there is an equivalence of $(\infty,2)$-categories
\[N(\Sigma\cP)\simeq \Sigma (N\cP)\]
between the nerve of the suspension and the suspension of the nerve.
  \end{thmalph}

 In \cref{sec:WedgeResults} we introduce the \emph{wedge} of two $n$- or $(\infty,n)$-categories, a particular way of gluing along an object, and we show in \cref{sec:WedgeProofs} the following compatibility of nerve and wedge, which will appear as \cref{ThmA}.

 \begin{thmalph}
 \label{ThmAIntro}
In the model of saturated $n$-complicial sets, for any $n$-categories $\cA$ and $\cA'$ there is an equivalence of $(\infty,n)$-categories
\[N(\cA\vee\cA')\simeq N\cA\vee N\cA'\]
between the nerve of their wedge and the wedge of their nerves.
 \end{thmalph}

As anticipated, we now elaborate on how \cref{ThmAIntro,ThmBIntro} can be used to then show that the nerve embedding preserves certain valuable pushouts.

   For $0\le m\le n$, an $m$-morphism of an $n$-category $\cD$ is represented by a functor $C_m\to\cD$, where $C_m$ is the free $m$-cell, so one can regard all free cells as the building blocks of $n$-categories. For instance, the free $0$-, $1$- and $2$-cells can be depicted as
  \begin{equation*}
  \begin{tikzpicture}[baseline={([yshift=-0.1cm]current bounding box.center)}, scale=1.0]
\def\l{0.7cm}
\def\s{2pt}
\def\r{2pt}
\begin{scope}[xshift=-7*\l]
\draw[fill] (-\l,0) circle (\r) node[inner sep=\s] (a){};
\draw (a){}+(-\l,0) node (C0){$C_0=$};
\end{scope}
\begin{scope}[xshift=-3.5*\l]
\draw[fill] (-\l,0) circle (\r) node[inner sep=\s] (a){};
\draw[fill] (\l,0) circle (\r) node[inner sep=\s] (b){};
\draw[-stealth] (a)--node[inner sep=\s](At){}(b);
\draw (a){}+(-\l,0) node (C1){$C_1=$};
\end{scope}
\begin{scope}[xshift=2*\l]
\draw[fill] (-\l,0) circle (\r) node[inner sep=\s] (a){};
\draw[fill] (\l,0) circle (\r) node[inner sep=\s] (b){};
\draw[-stealth] (a) to[bend left=70] node[inner sep=\s](As){}(b);
\draw[-stealth] (a) to[bend right=70]node[inner sep=\s](At){}(b);
\node (A) at ($(As)!0.5!(At)$) {{\Large $\Downarrow$}};
\draw (a){}+(-\l,0) node (C2){$C_2=$};
\end{scope}
\end{tikzpicture}
.
  \end{equation*}
Composition operations are governed by pasting diagrams, which can be realized as certain pushouts of $n$-categories, which are instances of Barwick--Schommer-Pries' ``fundamental pushouts'' from \cite{BarwickSchommerPries}.
For instance, composition of $2$-morphisms along objects and along $1$-morphisms are encoded in the pushouts in $n\cat$
  \begin{equation} 
  \label{eq:FundPushout}
\begin{tikzpicture}[baseline={([yshift=-0.1cm]current bounding box.center)}, scale=1.0]
\def\l{0.7cm}
\def\s{2pt}
\def\r{2pt}
\begin{scope}[xshift=-1.5*\l, yshift=3*\l]
\draw[fill] (-\l,0) circle (\r) node[inner sep=\s] (a){};
\end{scope}
\begin{scope}[xshift=-3.5*\l]
\draw[fill] (-\l,0) circle (\r) node[inner sep=\s] (a){};
\draw[fill] (\l,0) circle (\r) node[inner sep=\s] (b){};
\draw[-stealth] (a) to[bend left=70] node[inner sep=\s](As){}(b);
\draw[-stealth] (a) to[bend right=70]node[inner sep=\s](At){}(b);
\node (A) at ($(As)!0.5!(At)$) {{\Large $\Downarrow$}};
\end{scope}
\begin{scope}[yshift=3*\l, xshift=2*\l]
\draw[fill] (-\l,0) circle (\r) node[inner sep=\s] (a){};
\draw[fill] (\l,0) circle (\r) node[inner sep=\s] (b){};
\draw[-stealth] (a) to[bend left=70] node[inner sep=\s](As){}(b);
\draw[-stealth] (a) to[bend right=70]node[inner sep=\s](At){}(b);
\node (A) at ($(As)!0.5!(At)$) {{\Large $\Downarrow$}};
\end{scope}
\begin{scope}
\draw[fill] (-\l,0) circle (\r) node[inner sep=\s] (a){};
\draw[fill] (\l,0) circle (\r) node[inner sep=\s] (b){};
\draw[fill] (3*\l,0) circle (\r) node[inner sep=\s] (c){};
\draw[-stealth] (a) to[bend left=70] node[inner sep=\s](As){}(b);
\draw[-stealth] (a) to[bend right=70]node[inner sep=\s](At){}(b);
\draw[-stealth] (b) to[bend left=70] node[inner sep=\s](Bs){}(c);
\draw[-stealth] (b) to[bend right=70]node[inner sep=\s](Bt){}(c);
\node (A) at ($(As)!0.5!(At)$) {{\Large $\Downarrow$}};
\node (B) at ($(Bs)!0.5!(Bt)$) {{\Large $\Downarrow$}};
\end{scope}
\draw[right hook-stealth,thick] (-1.3*\l, 3*\l)--(0.1*\l, 3*\l);
\draw[right hook-stealth,thick] (-2.1*\l, 0)--(-1.3*\l, 0);
\draw[right hook-stealth,thick] (-2.5*\l, 2.2*\l)--(-2.5*\l, 0.9*\l);
\draw[right hook-stealth,thick] (1.7*\l, 2.2*\l)--(1.7*\l, 0.9*\l);
\end{tikzpicture}
\mbox{\quad and \quad}
\begin{tikzpicture}[baseline={([yshift=0.1cm]current bounding box.center)}, scale=1.0]
\def\l{0.7cm}
\def\s{2pt}
\def\r{2pt}
\def\h{0cm}
\begin{scope}[xshift=-3.5*\l, yshift=3*\l]
\draw[fill] (-\l,0) circle (\r) node[inner sep=\s] (a){};
\draw[fill] (\l,0) circle (\r) node[inner sep=\s] (b){};
\draw[-stealth] (a)--node[inner sep=\s](At){}(b);
\end{scope}
\begin{scope}[xshift=-3.5*\l]
\draw[fill] (-\l,-\h) circle (\r) node[inner sep=\s] (a){};
\draw[fill] (\l,-\h) circle (\r) node[inner sep=\s] (b){};
\draw[-stealth] (a) to[bend left=70] node[inner sep=\s](A){}(b);
\draw[-stealth] (a)--node[inner sep=\s](B){}(b);
\node (AB) at ($(A)!0.5!(B)$) {{\scriptsize $\Downarrow$}};
\end{scope}
\begin{scope}[yshift=3*\l]
\draw[fill] (-\l,-\h) circle (\r) node[inner sep=\s] (a){};
\draw[fill] (\l,-\h) circle (\r) node[inner sep=\s] (b){};
\draw[-stealth] (a) to[bend right=70]node[inner sep=\s](C){}(b);
\draw[-stealth] (a)--node[inner sep=\s](B){}(b);
\node (BC) at ($(B)!0.5!(C)$) {{\scriptsize $\Downarrow$}};
\end{scope}
\begin{scope}
\draw[fill] (-\l,0) circle (\r) node[inner sep=\s] (a){};
\draw[fill] (\l,0) circle (\r) node[inner sep=\s] (b){};
\draw[-stealth] (a) to[bend left=70] node[inner sep=\s](A){}(b);
\draw[-stealth] (a) to[bend right=70]node[inner sep=\s](C){}(b);
\draw[-stealth] (a)--node[inner sep=\s](B){}(b);
\node (AB) at ($(A)!0.5!(B)$) {{\scriptsize $\Downarrow$}};
\node (BC) at ($(B)!0.5!(C)$) {{\scriptsize $\Downarrow$}};
\end{scope}
\draw[right hook-stealth,thick] (-2.1*\l, 3*\l)--(-1.3*\l, 3*\l);
\draw[right hook-stealth,thick] (-2.1*\l, 0)--(-1.3*\l, 0);
\draw[right hook-stealth,thick] (-3.5*\l, 2.2*\l)--(-3.5*\l, 0.9*\l);
\draw[right hook-stealth,thick] (0.0*\l, 2.2*\l)--(0.0*\l, 0.9*\l);
\end{tikzpicture}
\tag{$*$}
  \end{equation}

In the new setup $(\infty,n)\cat$, one can make sense of cells and shapes obtained as the fundamental pushouts from (\ref{eq:FundPushout}) as $(\infty,n)$-categories.
 Cells should still detect morphisms, and the pushouts should still encode composition of morphisms. However, for this to be meaningful, the fundamental pushouts regarded as $(\infty,n)$-categories must be also the resulting pushout in $(\infty,n)$-categories.
It is therefore expected, and included in the axioms for a model of $(\infty,n)$-categories in the sense of \cite{BarwickSchommerPries}, that the fundamental pushouts are preserved by the embedding.

  Using \cref{ThmAIntro,ThmBIntro} we can show that the fundamental pushouts from (\ref{eq:FundPushout}) are preserved in the model of $n$-complicial sets,
  providing in particular a first step towards proving the equivalence of saturated $n$-complicial sets with other models.

 More precisely, as an instance of \cref{ThmAIntro}, we obtain the following corollary, asserting the preservation of the first fundamental pushout from (\ref{eq:FundPushout}), which will appear as \cref{CorA}.
  
  \begin{coralph}
  \label{CorAIntro}
In the model of saturated $n$-complicial sets, there is an equivalence of $(\infty,2)$-categories
\[
N\left(
\begin{tikzpicture}[baseline={([yshift=-0.1cm]current bounding box.center)}, scale=1.0]
\def\l{0.7cm}
\def\s{2pt}
\def\r{2pt}
\draw[fill] (-\l,0) circle (\r) node[inner sep=\s] (a){};
\draw[fill] (\l,0) circle (\r) node[inner sep=\s] (b){};
\draw[fill] (3*\l,0) circle (\r) node[inner sep=\s] (c){};
\draw[-stealth] (a) to[bend left=70] node[inner sep=\s](As){}(b);
\draw[-stealth] (a) to[bend right=70]node[inner sep=\s](At){}(b);
\draw[-stealth] (b) to[bend left=70] node[inner sep=\s](Bs){}(c);
\draw[-stealth] (b) to[bend right=70]node[inner sep=\s](Bt){}(c);
\node (A) at ($(As)!0.5!(At)$) {{\scriptsize {\Large $\Downarrow$}}};
\node (B) at ($(Bs)!0.5!(Bt)$) {{\scriptsize {\Large $\Downarrow$}}};
\end{tikzpicture}
\right) \simeq N\left(
\begin{tikzpicture}[baseline={([yshift=-0.1cm]current bounding box.center)}, scale=1.0]
\def\l{0.7cm}
\def\s{2pt}
\def\h{0cm}
\def\r{2pt}
\draw[fill] (-\l,-\h) circle (\r) node[inner sep=\s] (a){};
\draw[fill] (\l,-\h) circle (\r) node[inner sep=\s] (b){};
\draw[-stealth] (a) to[bend left=70] node[inner sep=\s](A){}(b);
\draw[-stealth] (a) to[bend right=70] node[inner sep=\s](B){}(b);
\node (AB) at ($(A)!0.5!(B)$) {{\scriptsize {\Large $\Downarrow$}}};
\end{tikzpicture}
\right)
%
%
\aamalg{N\left(
\begin{tikzpicture}[baseline={([yshift=-0.05cm]current bounding box.center)}, scale=0.5]
\def\l{0.7cm}
\def\s{2pt}
\def\r{4pt}
\draw[fill] (-\l,0) circle (\r) node[inner sep=\s] (a){};
\end{tikzpicture}
\right)}
N\left(
\begin{tikzpicture}[baseline={([yshift=-0.1cm]current bounding box.center)}, scale=1.0]
\def\l{0.7cm}
\def\s{2pt}
\def\h{0cm}
\def\r{2pt}
\draw[fill] (-\l,-\h) circle (\r) node[inner sep=\s] (a){};
\draw[fill] (\l,-\h) circle (\r) node[inner sep=\s] (b){};
\draw[-stealth] (a) to[bend right=70]node[inner sep=\s](C){}(b);
\draw[-stealth] (a) to[bend left=70]node[inner sep=\s](B){}(b);
\node (BC) at ($(B)!0.5!(C)$) {{\scriptsize {\Large $\Downarrow$}}};
\end{tikzpicture}
\right).
\]

  \end{coralph}

Secondly, using \cref{ThmBIntro}, we prove the following corollary, which asserts the preservation of the second fundamental pushout from (\ref{eq:FundPushout}), and will appear as \cref{CorB}.

  \begin{coralph}
  \label{CorBIntro}
  In the model of saturated $2$-complicial sets, there is an equivalence of $(\infty,2)$-categories
\[
N\left(
\begin{tikzpicture}[baseline={([yshift=-0.1cm]current bounding box.center)}, scale=1.0]
\def\l{0.7cm}
\def\s{2pt}
\def\r{2pt}
\draw[fill] (-\l,0) circle (\r) node[inner sep=\s] (a){};
\draw[fill] (\l,0) circle (\r) node[inner sep=\s] (b){};
\draw[-stealth] (a) to[bend left=70] node[inner sep=\s](A){}(b);
\draw[-stealth] (a) to[bend right=70]node[inner sep=\s](C){}(b);
\draw[-stealth] (a)--node[inner sep=\s](B){}(b);
\node (AB) at ($(A)!0.5!(B)$) {{\scriptsize $\Downarrow$}};
\node (BC) at ($(B)!0.5!(C)$) {{\scriptsize $\Downarrow$}};
\end{tikzpicture}
\right) \simeq N\left(
\begin{tikzpicture}[baseline={([yshift=-0.4cm]current bounding box.center)}, scale=1.0]
\def\l{0.7cm}
\def\s{2pt}
\def\h{0cm}
\def\r{2pt}
\draw[fill] (-\l,-\h) circle (\r) node[inner sep=\s] (a){};
\draw[fill] (\l,-\h) circle (\r) node[inner sep=\s] (b){};
\draw[-stealth] (a) to[bend left=70] node[inner sep=\s](A){}(b);
\draw[-stealth] (a)--node[inner sep=\s](B){}(b);
\node (AB) at ($(A)!0.5!(B)$) {{\scriptsize $\Downarrow$}};
\end{tikzpicture}
\right)
%
%
\aamalg{N\left(
\begin{tikzpicture}[baseline={([yshift=-0.05cm]current bounding box.center)}, scale=0.5]
\def\l{0.7cm}
\def\s{2pt}
\def\r{4pt}
\draw[fill] (-\l,0) circle (\r) node[inner sep=\s] (a){};
\draw[fill] (\l,0) circle (\r) node[inner sep=\s] (b){};
\draw[-stealth] (a)--node[inner sep=\s](B){}(b);
\end{tikzpicture}
\right)}
N\left(
\begin{tikzpicture}[baseline={([yshift=0.3cm]current bounding box.center)}, scale=1.0]
\def\l{0.7cm}
\def\s{2pt}
\def\h{0cm}
\def\r{2pt}
\draw[fill] (-\l,-\h) circle (\r) node[inner sep=\s] (a){};
\draw[fill] (\l,-\h) circle (\r) node[inner sep=\s] (b){};
\draw[-stealth] (a) to[bend right=70]node[inner sep=\s](C){}(b);
\draw[-stealth] (a)--node[inner sep=\s](B){}(b);
\node (BC) at ($(B)!0.5!(C)$) {{\scriptsize $\Downarrow$}};
\end{tikzpicture}
\right).
\]
  \end{coralph}
  
  Finally, \cref{ThmBIntro} also yields the following corollary, which will appear as \cref{CorC}. It asserts that the nerve embedding preserves the equivalence between the free $1$-cell $C_1$ and the free living $2$-isomorphism. This is another condition that in other models is encoded into a completeness axiom, and is instead combinatorially involved in the model of saturated complicial sets.
  
  \begin{coralph}
    \label{CorCIntro}
In the model of saturated $2$-complicial sets, there is an equivalence of $(\infty,2)$-categories
    \[
N\left(
\begin{tikzpicture}[baseline={([yshift=-0.1cm]current bounding box.center)}, scale=1.0]
\def\l{0.7cm}
\def\s{2pt}
\def\r{2pt}
\draw[fill] (-\l,0) circle (\r) node[inner sep=\s] (a){};
\draw[fill] (\l,0) circle (\r) node[inner sep=\s] (b){};
\draw[-stealth] (a) to[bend left=70] node[inner sep=\s](As){}(b);
\draw[-stealth] (a) to[bend right=70]node[inner sep=\s](At){}(b);
\node (A) at ($(As)!0.5!(At)$) {{\scriptsize { {\Large $\Downarrow$} $\cong$}}};
\end{tikzpicture}
\right) \simeq N\left(
\begin{tikzpicture}[baseline={([yshift=-0.1cm]current bounding box.center)}, scale=1.0]
\def\l{0.7cm}
\def\s{2pt}
\def\h{0cm}
\def\r{2pt}
\draw[fill] (-\l,-\h) circle (\r) node[inner sep=\s] (a){};
\draw[fill] (\l,-\h) circle (\r) node[inner sep=\s] (b){};
\draw[-stealth] (a)--(b);
\end{tikzpicture}
\right).
\]
  \end{coralph}
  
  \addtocontents{toc}{\protect\setcounter{tocdepth}{1}}
  \subsection*{Acknowledgements} 
  We would like to thank Clark Barwick and Chris Schommer-Pries for explanations concerning fundamental pushouts, and Len\-nart Meier and Emily Riehl for helpful conversations on this project.
{This material is based upon work supported by the National Science Foundation under Grant No.\ DMS-1440140 while the authors were in residence at the Mathematical Sciences Research Institute in Berkeley, California, during the Spring 2020 semester. The first-named author thankfully acknowledges the financial support by the DFG grant OZ 91/2-1 with the project nr.~442418934.}
  \newpage
  
    \tableofcontents
    
   \section{Background on $n$-complicial sets}
  We assume the reader to be familiar with the basics of strict higher category theory (see e.g.~\cite{LeinsterHigherOperads}) and with the model categorical language (see e.g.~\cite{Hirschhorn,hovey}), and we recall the preliminary material that will be used in the paper.

  The category $n\cat$ of $n$-categories is defined recursively as the category of categories enriched over the category of $(n-1)$-categories, assuming that the category of $0$-categories is the category $\set$ of sets with the cartesian product.
In particular, an \emph{$n$-category} $\cD$ consists of a set of objects and for any objects $x,x'$ an $(n-1)$-category $\Map_{\cD}(x,x')$, together with a horizontal composition that defines a functor of hom-$(n-1)$-categories
$\circ\colon\Map_{\cD}(x,x')\times\Map_{\cD}(x',x'')\to\Map_{\cD}(x,x'').$
 For $n=\infty$, the convention above specializes to an \emph{$\omega$-category}, as in \cite{StreetOrientedSimplexes,VerityComplicialAMS}.

The following model structure models the standard homotopy theory of $n$-categories. It recovers the canonical model structure for $1$-categories as well as Lack's model structure for $2$-categories from \cite{lack1}.

  \begin{thm}[{\cite[Thm 5]{LMW}}]
\label{modelstructureondiscretepresheaves}
Let $n\in\mathbb N\cup\{\infty\}$. The category $n\cat$ supports a cofibrantly generated model structure in which
\begin{itemize}[leftmargin=*]
\item all $n$-categories are fibrant;
\item the weak equivalences are precisely the $n$-categorical equivalences.
\end{itemize}
\end{thm}

In this paper, we will consider a model of $(\infty,n)$-categories due to Verity based on the following mathematical object.

\begin{defn}
A \emph{simplicial set with marking}\footnote{Originally referred to as \emph{simplicial set with hollowness} in \cite{StreetOrientedSimplexes} and later as \emph{stratified simplicial set} e.g.~in \cite{VerityComplicialAMS}.} 
is a simplicial set endowed with a subset of simplices of strictly positive dimensions that contain all degenerate simplices, called \emph{thin} or \emph{marked}. We denote by $m\sset$ the category of simplicial sets with marking and marking preserving simplicial maps.
\end{defn} 

\begin{rmk}
The underlying simplicial set functor $m\sset\to\sset$ respects limits and colimits, since it is both a left and a right adjoint (see e.g.~\cite[Obs.~ 97]{VerityComplicialAMS}), and it preserves and reflects monomorphisms, since it is a faithful right adjoint.
 Moreover, as explained in \cite [Obs.~109]{VerityComplicialAMS},
\begin{itemize}[leftmargin=*]
\item 
 a simplex is marked in a limit of simplicial sets with marking $\lim_{i\in I}X_i$ if and only  if it is marked in each component $X_i$ for $i\in I$, and
 \item a simplex is marked in a colimit of simplicial sets with marking $\colim_{i\in I}X_i$ if and only if it admits a marked representative in $X_i$ for some $i\in I$.
 \end{itemize}
\end{rmk}

The following model structure provides a model for the homotopy theory of $(\infty,n)$-categories. It is obtained applying Verity's machinery \cite[\textsection 6.3]{VerityComplicialAMS} to a special set of anodyne extensions, described in \cite{EmilyNotes} and recalled in \cref{anodynemaps}.

\begin{thm}[{\cite[Thm 1.28]{or}}]
Let $n\in\mathbb N\cup\{\infty\}$.
\label{modelstructureondiscretepresheaves}
The category $m\sset$ supports a cofibrantly generated left proper model cartesian structure where
\begin{itemize}[leftmargin=*]
\item the fibrant objects are precisely the \emph{saturated $n$-complicial sets}, i.e., those with the right lifting property with respect to the elementary anodyne extensions, recalled in \cref{anodynemaps};
\item the cofibrations are precisely the monomorphisms (of underlying simplicial sets).
\end{itemize}
We call this model structure the \emph{model structure for $(\infty,n)$-categories}, and we call the weak equivalences the $(\infty,n)$-weak equivalences.
\end{thm}

The interpretation is that, in a saturated $n$-complicial sets, the marked $k$-simplices are precisely the $k$-equivalences. We refer the reader e.g.~to \cite{EmilyNotes} for further elaboration on this viewpoint.

In order to recall the elementary anodyne extensions, we need also the following preliminary terminology and notation.

\begin{defn}
A sub-simplicial set with marking $X$ of a simplicial set with marking $Y$ is \emph{regular} if a simplex of $X$ is marked in $X$ if and only if it is marked in $Y$.
\end{defn}

\begin{notn} 
\label{preliminarynotation}
We denote
\begin{itemize}[leftmargin=*]
    \item by $\Delta[m]$ the standard $m$-simplex in which exactly the degenerate simplices are marked;
       \item by $\Delta[m]_{t}$ the standard $m$-simplex in which the only marked non-degenerate simplex is the top-dimensional one;
    \item by $\Delta^k[m]$, for $0\leq k \leq m$, the standard $m$-simplex in which a non-degenerate simplex is marked if and only if it contains the vertices $\{k-1,k,k+1\}\cap [m]$;
    \item by $\Delta^k[m]'$, for $0\leq k \leq m$, the standard $m$-simplex with marking obtained from $\Delta^k[m]$ by additionally marking the $(k-1)$-st and $(k+1)$-st face of $\Delta[m]$;
    \item by $\Delta^k[m]''$, for $0\leq k \leq m$, the standard $m$-simplex with obtained from $\Delta^k[m]'$ by additionally marking the $k$-th face of $\Delta[m]$;
    \item by $\Lambda^k[m]$, for $0\leq k \leq m$, the regular sub-simplicial set of $\Delta^k[m]$ with marking whose simplicial set is the $k$-horn $\Lambda^k[m]$;
    \item by $\eqDelta$ the $3$-simplex in which the non-degenerate marked simplices consist of all $2$- and $3$-simplices, as well as $1$-simplices $[02]$ and $[13]$;
        \item by $\Delta[3]^\sharp$ the $3$-simplex in which all simplices in positive dimensions are marked.
\end{itemize}
\end{notn}

\begin{defn}
\label{anodynemaps}
Let $n\in\mathbb N\cup\{\infty\}$.
An \emph{$(\infty,n)$-elementary anodyne extension} is one of the following maps of simplicial sets with marking.
 \begin{enumerate}[leftmargin=*]
  \item  The \emph{complicial horn extension}, i.e., the canonical map
 $$\Lambda^k[m]\to \Delta^k[m]\text{ for $m\geq 1$ and $0\leq k\leq m$},$$
 which is an ordinary horn inclusion on the underlying simplicial sets.
 \item The \emph{thinness extension}, i.e., the canonical map 
$$\Delta^k[m]' \to \Delta^k[m]''\text{ for $m\geq 2$ and $0\leq k \leq m$},$$
which is an identity on the underlying simplicial set.
\item The \emph{triviality extension} map,
i.e., the canonical map
$$\Delta[l]\to \Delta[l]_t\text{ for $l>n$},$$
which is an identity on the underlying simplicial set.
\item The \emph{saturation extension}\footnote{Note that the last condition was phrased slightly different in \cite{or}, namely we used as elementary saturation anodyne extensions the maps
$\Delta[l]\star \Delta[3]_{eq}  \to \Delta[l]\star \Delta[3]^{\sharp}\text{ for $l\geq -1$}$.
As a consequence of the discussion following \cite[Def.~ D.7.9]{RiehlVerityBook}, the model structures resulting from both conditions are equal (in the presence of the remaining elementary anodyne extensions). We chose to work with this convention to simplify the proof of \cref{conehomotopical}.}, i.e., the canonical map
\[
\Delta[3]_{eq}\star \Delta[l] \to  \Delta[3]^{\sharp}\star \Delta[l]\text{ for $l\geq -1$}
\]
which is an identity on the underlying simplicial set.
Here, the construction $\star$ denotes the join construction of simplicial sets with marking, which is recalled in \cref{DefJoin}.
\end{enumerate}
\end{defn}

Although there is no explicit description of generating acyclic cofibrations for this model structure, the elementary anodyne extensions provide a good approximation, in the sense of the following lemma.

\begin{lem}
\label{lemmaleftquillen}
A functor $F\colon m\sset\to\cM$
is left Quillen when $m\sset$ is endowed with the model category for $(\infty,n)$-categories and $\cM$ is any model category if and only if $F$ is a left adjoint, it respects cofibrations and sends all elementary  anodyne extensions from \cref{anodynemaps} to weak equivalences of $\cM$.
\end{lem}

\begin{proof}
By Cisinski--Olschok theory (see e.g.~\cite[Theorem 3.16, Lemma 3.30]{Olschok}),
one can show that the fibrations between fibrant objects in the model structure for $(\infty,n)$-categories are precisely the maps having the right lifting property with respect to the elementary anodyne extensions from \cref{anodynemaps}. By adjointness, if $F$ is a left adjoint functor that respects cofibrations, it sends elementary anodyne extensions to weak equivalences if and only if the right adjoint preserves fibrations between fibrant objects. By \cite[Proposition 7.15]{JT}, this is equivalent to saying that $F$ is a left Quillen functor, as desired.
\end{proof}

As a special case of the slice model structures, constructed e.g.~in \cite{HirschhornOvercategories}, we also obtain model structure on the category $m\sset_{*}$ of pointed simplicial sets with marking and on the category $m\sset_{*,*}$ of bi-pointed simplicial sets with marking.

\begin{prop}
\label{pointedmodelstructure}
The category $m\sset_*$, resp. $m\sset_{*,*}$, supports a cofibrantly generated left proper model structure where
\begin{itemize}[leftmargin=*]
\item the fibrant objects are precisely the pointed, resp.~bipointed, simplicial sets with marking whose underlying simplicial sets with marking are saturated $n$-complicial sets.
\item the cofibrations are precisely the monomorphisms (on underlying simplicial sets).
\end{itemize}
We call this model structure the \emph{model structure for pointed $(\infty,n)$-categories}, resp.~the \emph{model structure for bi-pointed $(\infty,n)$-categories}.
\end{prop}

We fix the following terminology.

\begin{defn}
A map of simplicial sets with marking $X\to Y$ is a \emph{complicial inner anodyne extension} if it can be written as a retract of a transfinite composition of pushouts of maps of the following form:
\begin{enumerate}[leftmargin=*]
 \item  inner complicial horn extensions
 $$\Lambda^k[m]\to \Delta^k[m]\text{ for $m> 1$ and $0< k< m$},$$
 \item complicial thinness extensions
$$\Delta^k[m]' \to \Delta^k[m]''\text{ for $m\geq 2$ and $0\leq k \leq m$}.$$
\end{enumerate}
\end{defn}

\begin{rmk}
 \label{underlyingcomplicialinner}
One can prove with standard model categorical techniques the following formal properties of complicial inner anodyne extensions.
\begin{enumerate}[leftmargin=*]
\item Any complicial inner anodyne extension is an $(\infty,n)$-acyclic cofibration.
\item The underlying simplicial map of a complicial inner anodyne extension is an inner anodyne extension of simplicial sets.
     \item The class of complicial inner anodyne extensions is closed under transfinite composition and pushouts.
\end{enumerate}
\end{rmk}

We will produce several complicial inner anodyne extensions using the following one.

\begin{lem}
 \label{CompMarkAtOnce} For $m\geq 2$ and $0<k<m$, let $\Lambda^k[m]'$ denote the regular subset of $\Delta^k[m]'$ whose underlying simplicial set is given by the $k$-horn $\Lambda^k[m]$. The inclusion
    \[\Lambda^k[m]'\to \Delta^k[m]''\text{ for $m\geq 2$, $0<k<m$}\]
is a complicial inner anodyne extension.
\end{lem}

\begin{proof}
The desired inclusion can be written as a composite
\[\Lambda^k[m]' \hookrightarrow \Delta^k[m]' \hookrightarrow \Delta^k[m]''\]
where the second arrow is a thinness anodyne extension and the first arrow is a pushout of an elementary inner complicial inner horn extension
\[
\begin{tikzcd}
\Lambda^k[m] \arrow[r] \arrow[d, hook]&\Lambda^k[m]' \arrow[d, hook] \\
\Delta^k[m] \arrow[r] & \Delta^k[m]'.
\end{tikzcd}
\]
This proves the claim.
\end{proof}

For any $n$-category $\cD$, Street \cite{StreetOrientedSimplexes} defined a simplicial nerve $N\cD$ in terms of the $n$-truncated orientals $\mathcal O_n[m]$. The $n$-category $\mathcal O_n[m]$ should be thought as the free $n$-category over an $m$-simplex. For a precise account on orientals we refer the reader to \cite{StreetOrientedSimplexes} or \cite[\textsection 7]{AraMaltsiniotisJoin}.

When $n=2$, we will make use of the following explicit description of the $2$-truncated oriental.

\begin{defn} Let $m\ge0$.
\label{orientals}
The \emph{$2$-truncated $m$-oriental}
is the $2$-category $\cO_2[m]$ in which
\begin{enumerate}[leftmargin=*]
\item[(0)] there are $m+1$ objects $x_0,\dots,x_{m}$;
    \item the $1$-morphisms are freely generated under composition by the $1$-morphisms $f_{ij}\colon x_i\to x_j$ for $i\le j$;
    \item the $2$-morphisms are generated under composition by the $2$-morphisms $\alpha_{ijk}\colon f_{ik}\Rightarrow f_{jk}\circ f_{ij}$ for $i<j<k$, subject to the relations that
    for any $i< j< k < s$ 
    \[
    (\id_{f_{ks}}\circ_h\alpha_{ijk})\circ_v\alpha_{iks}= (\alpha_{jks}\circ_h \id_{f_{ij}})\circ_v \alpha_{ijs}.
    \]
 \end{enumerate}
\end{defn}

\begin{rmk}
\label{RmkMathfrakC}
When regarded as a simplicial category, $\cO_2[m]$ is isomorphic to $\mathfrak C[\Delta[m]]$, the homotopy coherent realization of the standard simplex, as studied in \cite[Def.~1.1.5.1]{htt}.
In particular, there we find the following alternative description. For any $0\le i,j\le m$ the hom-category  $\Map_{\cO_2[m]}(x_i,x_j)$ is given by
    \[\Map_{\cO_2[m]}(x_i,x_j):=\left\{
    \begin{array}{ll}
            [1]^{j-i-1} & j> i \\
            {[0]} & j=i\\
       \varnothing  &  j<i.
    \end{array}\right.\]
    This can be reformulated further saying that each $1$-morphism of $\cO_2[m]$ from $x_i$ to $x_j$ is uniquely represented as a subset of $\{i, i+1,\dots,j-1, j\}$ containing $i$ and $j$, and each $2$-morphism is uniquely represented as an inclusion of such subsets.
\end{rmk}

The geometry of orientals is such that one can define the following nerve.

\begin{defn}
Let $n\in\mathbb N\cup\{\infty\}$.
The \emph{Street nerve} $N\cD$ of an $n$-category $\cD$ is the simplicial set in which
\begin{itemize}[leftmargin=*]
    \item an $m$-simplex is an $n$-functor $\mathcal O_n[m]\to \cD$.
    \item the simplicial structure is induced by the geometry of orientals.
\end{itemize}
\end{defn}

For $n=2$, the Street nerve was studied in detail by Duskin in \cite{duskin}, and can be described explicitly as follows.
 
\begin{defn}
The \emph{nerve} $N\cD$ of a $2$-category $\cD$ is the $3$-coskeletal simplicial set in which
\begin{enumerate}[leftmargin=*]
\item[(0)] a $0$-simplex consists of an object of $\cD$:
$$x;$$
    \item a $1$-simplex consists of a $1$-morphism of $\cD$:
     \[
\begin{tikzcd}
    x \arrow[rr, "a"{below}]&& y;
\end{tikzcd}
\]
    \item a $2$-simplex consists of a $2$-cell of $\cD$ of the form $c\Rightarrow b\circ a$:
    $$\begin{tikzcd}[baseline=(current  bounding  box.center)]
 & y \arrow[rd, "b"]  & \\
    x \arrow[ru, "{a}"]
  \arrow[rr, "c"{below}, ""{name=D,inner sep=1pt}]
  && z;
  \arrow[Rightarrow, from=D, 
 to=1-2, shorten >= 0.1cm, shorten <= 0.1cm, ""]
\end{tikzcd}$$
\item a $3$-simplex consists of four $2$-cells of $\cD$ that satisfy the following relation.
\[
\simpfver{d}{c}{e}{a}{b}{f}{}{}{}\simpfvercontinued{}{x}{y}{z}{w}
\]
\end{enumerate}
and in which the simplicial structure is as indicated in the pictures.
\end{defn}

The Street nerve can be endowed with the following marking, originally considered by Roberts in unpublished work and Street in \cite{StreetOrientedSimplexes}, further studied by Verity in \cite{VerityComplicialAMS}, and later discussed by Riehl in \cite{EmilyNotes}.

\begin{defn}
Let $n\in\mathbb N\cup\{\infty\}$.
The \emph{Roberts--Street nerve} is the simplicial set with marking $\NRS\cD$, in which
\begin{itemize}[leftmargin=*]
    \item the underlying simplicial set is the Street nerve $N\cD$, and
    \item an $m$-simplex of $N\cD$ is marked in $\NRS\cD$ if and only if the corresponding $n$-functor $\cO_n[m]\to\cD$ sends the top-dimensional $m$-cell of $\cO_n[m]$ to an identity of $\cD$. In particular, all simplices in dimension at least $n+1$ are marked.
\end{itemize}
 \end{defn}

We will use the following pointset and homotopical properties of $\NRS$.

 \begin{prop} \label{NRSproperties}
The Roberts--Street nerve
\[N^{RS}\colon n\cat\to m\sset\]
\begin{itemize}[leftmargin=*]
    \item is a right adjoint functor, and in particular respects all limits;
    \item is a homotopical functor between the model structure for $n$-categories and the model structure for $(\infty,n)$-categories if $n\le2$.
\end{itemize}
 \end{prop}

 \begin{proof}
 The fact that $\NRS$ is a right adjoint can be found in \cite[\textsection 10.3]{VerityComplicialAMS}. We now argue that if $n\le2$ the functor $\NRS$ is a homotopical functor, using the following auxiliary construction, considered e.g.~in \cite[\textsection 3.2]{EmilyNotes}.
 
 Given any $n$-category for $n\le2$, one can consider the simplicial set with marking $\Nnat\cD$ in which the simplicial set is $N\cD$ and in which
 \begin{enumerate}[leftmargin=*]
        \item a $1$-simplex is marked in $\Nnat\cD$ if and only if the representing $1$-morphism in $\cD$ is an equivalence.
        \item a $2$-simplex is marked in $\Nnat\cD$ if and only if the representing $2$-morphism in $\cD$ is an isomorphism.
        \item all simplices of $\Nnat\cD$ in dimension $3$ or higher are marked.
    \end{enumerate}
    There is a natural inclusion of simplicial sets with marking $\NRS\cD\to\Nnat\cD$, which can be seen to be an $(\infty,n)$-weak equivalence combining \cite[Thm~5.2]{Nerves2Cat} and \cite[Prop.~1.31]{or}.
    The construction extends to a functor $\Nnat\colon n\cat\to m\sset$, which can be seen to be homotopical combining \cite[Thm~4.12]{Nerves2Cat} and \cite[Prop.~1.31]{or}.

 Now, suppose we are given a weak equivalence of $n$-categories $F\colon \cD\to \cD'$ for $n\le2$. It fits into the following commutative diagram 
\[{
\begin{tikzcd}
\NRS\cD\arrow[d, ""swap] \arrow[r, "\NRS F"]&\NRS\cD'\arrow[d, ""] \\
 N^{\natural}\cD \arrow[r, "\Nnat F"swap] & N^{\natural}\cD'.
\end{tikzcd}
}\]
By previous considerations, the vertical maps and the bottom map are equivalences of $(\infty,n)$-categories, so the top map must also be one.
 \end{proof}

   \section{Nerve vs suspension - The results}
   \label{sec:SuspensionResults}
   
  In this section, we illustrate the results and applications related to the compatibility of nerve and suspension constructions.
   
   We recall the $2$-categorical suspension\footnote{The $2$-categorical suspension $\Sigma\cD$ appears in \cite{BarwickSchommerPries} as $\sigma(\cD)$.   It also often appears in the literature as a special case of a simplicial suspension. For instance,
the homwise nerve $N_*(\Sigma\cD)$ of the suspension $\Sigma\cD$ is a simplicial category that agrees with what would be denoted as $U(N\cD)$ in \cite{bergner}, as $S(N\cD)$ in \cite{Joyal2007}, as $[1]_{N\cD}$ in \cite{htt}, 
and as ${\mathbbm{2}}[N\cD]$ in \cite{RiehlVerityNcoh}.}.

\begin{defn}
\label{suspension}
Let $\cD$ be a $1$-category $\cD$.
The \emph{suspension}
of $\cD$
is the $2$-category $\Sigma\cD$ in which
\begin{itemize}[leftmargin=*]
    \item[(a)] there are two objects $x_{\bot}$ and $x_{\top}$
    \item[(b)] the hom-$1$-categories given by 
\[
\Map_{\Sigma\cD}(a,b):=
\left\{
\begin{array}{cll}
\cD&\mbox{ if }a=x_{\bot},b=x_{\top}\\
{[0]}& \mbox{ if }a=b, \\
\varnothing& \mbox{ if } a=x_{\top},b=x_{\bot}
\end{array}
\right.
\]
\item[(c)] there is no nontrivial horizontal composition.
\end{itemize} 
\end{defn}

 \begin{ex} Let $k,l\ge0$.
\begin{itemize}[leftmargin=*]
    \item The suspension $\Sigma[k]$ of the poset $[k]$ is the free $k$-tuple of vertically composable $2$-morphisms, namely the $2$-category $[1|k]$ belonging to Joyal's cell category $\Theta_2$.
    \item The suspension $\Sigma([k]\times[l]^{\op})$ of the poset $[k]\times[l]^{\op}$ can be understood as a quotient of the $2$-truncated oriental $\cO_2[k+1+l]$ as explained by the following proposition.\footnote{Part of the arguments are inspired by \cite[\textsection 4]{VerityComplicialII} and \cite{AraMaltsiniotisVers}.}
    \item The suspension $\Sigma\mathbb I$ of the free isomorphism $\mathbb I$ is the walking $2$-isomorphism.
\end{itemize}
 \end{ex}
 
  \begin{prop}
  \label{quotientoforiental} 
  For any $k,l\geq -1$ there is a natural isomorphism of $2$-categories
\[\Sigma([k]\times [l]^{\op})\cong{}_{\cO_2[k]}\backslash\cO_2[k+1+l]/_{\cO_2[l]}\]
between the suspension of the poset $[k]\times[l]^{\op}$ and the quotient
${}_{\cO_2[k]}\backslash\cO_2[k+1+l]/_{\cO_2[l]}$ of the $2$-truncated $(k+1+l)$-oriental $\cO_2[k+1+l]$ obtained by collapsing $\cO_2[k]\cong \cO_2[\{0,\dots,k\}]\hookrightarrow \cO_2[k+1+l]$ to one point and $\cO_2[l]\cong\cO_2[\{k+1,\dots,k+1+l\}]\hookrightarrow\cO_2[k+1+l]$ to a different point.
\end{prop}

\begin{proof}
We define a $2$-functor
\[\varphi\colon\cO_2[k+1+l] \to \Sigma([k]\times [l]^{\op})\]
that is natural in $k$ and $l$ using the description of orientals in terms of objects, generating $1$- and $2$-morphisms as discussed in \cref{orientals} and the description of $1$- and $2$-morphisms of $\Sigma([k]\times [l]^{\op})$ as objects and $1$-morphisms of $[k]\times [l]^{\op}$.
\begin{enumerate}[leftmargin=*]
    \item[(a)] On objects, we set for any $0\le i\le k+1+l$
    \[\varphi(x_i):=\left\{
    \begin{array}{ll}
        x_{\bot} &  \text{if $0\le i\le k$}\\
        x_{\top} & \text{if $k+1\le i$.}
    \end{array}\right.\]
    \item[(b)]  On generating $1$-morphisms, we set for any $0\le i<j\le k+1+l$
\[
\varphi(f_{ij}):=\left\{
\begin{array}{ll}
\id_{x_{\bot}} &\mbox{ if } 0\leq i<j \leq k,\\
(i,j-k-1) &\mbox{ if } 0\leq i\leq k < j,\\
\id_{x_{\top}} &\mbox{ if } k<i<j.
\end{array}
\right.
\]
    \item[(c)]
On generating $2$-morphisms, we set for any $0\le i<j<s\le k+1+l$
\[
\varphi(\alpha_{ijs}):=\left\{ \begin{array}{ll}
\id_{\id_{x_{\bot}}}& \mbox{ if } 0\leq i<j<s \leq k,\\
(i,s-k-1)<(j,s-k-1)& \mbox{ if } 0\leq i<j \leq k<s,\\
(i,s-k-1)<(i,j-k-1)& \mbox{ if } 0\leq i \leq k<j<s,\\
\id_{\id_{x_{\top}}}& \mbox{ if } k<i<j<s.
\end{array}
\right.
\]
\end{enumerate}
To see that $\varphi$ is well-defined on $2$-morphisms and functorial, it is enough to observe that $\Sigma([k]\times[l]^{\op})$ is a category enriched in posets, and any two $2$-morphisms with the same source and target must coincide. By inspection, the $2$-functor $\varphi$ is also natural in both $k$ and $l$.

The $2$-functor $\varphi$ induces a $2$-functor
\[\widetilde{\varphi}\colon{}_{\cO_2[k]}\backslash\cO_2[k+1+l]/_{\cO_2[l]}\to \Sigma([k]\times [l]^{\op}),\]
and we argue that it is the desired isomorphism of $2$-categories.
\begin{enumerate}[leftmargin=*]
    \item[(0)] The $2$-functor $\widetilde{\varphi}$ is bijective on objects by construction.
    \item The $2$-functor $\widetilde{\varphi}$ is bijective on $1$-morphisms. Indeed, a careful inspection shows that the non-identity $1$-morphisms of ${}_{\cO_2[k]}\backslash\cO_2[k+1+l]/_{\cO_2[l]}$ are represented uniquely by $f_{i_1i_2}$ for $i_1\leq k <i_2$, and essentially by definition the $1$-morphisms of $\Sigma([k]\times[l]^{\op})$ are uniquely described as $(i_1,i_2-k-1)$ for $i_1\leq k <i_2$.
     \item The $2$-functor $\widetilde{\varphi}$ is bijective on $2$-morphisms.
     To see this, recall from \cref{RmkMathfrakC} 
     that each $2$-morphism of $\cO_2[k+1+l]$ from $0$ to $k+1+l$ is uniquely represented as a $1$-morphism of the poset $\cP(\{0, 1,\dots,k+l, k+1+l\})$ between subsets containing $0$ and $k+1+l$. Following this viewpoint, each non-identity $2$-morphism of ${}_{\cO_2[k]}\backslash\cO_2[k+1+l]/_{\cO_2[l]}$ from $x_{\bot}$ to $x_{\top}$ is uniquely represented as a $1$-morphism of the poset $\cP(\{0, 1,\dots,k+l, k+1+l\})$ 
     of the form
     \[
     \begin{tikzcd}
     \big\{0,1, \ldots , i_1-1, i_1, i_2, i_2+1,\ldots k+l, k+1+l\}\arrow[d,hook] \\
     \{0,1, \ldots , i_1'-1, i_1', i_2', i_2'+1,\ldots,k+l,k+1+l\big\}
     \end{tikzcd}
     \]
     with $i_1\leq k< i_2$ and $i_1'\leq k <i_2'$. In particular, $i_1 \leq i_1'$ and $i_2'\leq i_2$. By inspection, such $2$-morphism is sent by $\widetilde{\varphi}$ to the $2$-morphism of $\Sigma([k]\times[l]^{\op})$
represented by the $1$-morphism of $[k]\times[l]^{\op}$
     \[
     \begin{tikzcd}
     (i_1, i_2-k-1)\arrow[d] \\
     (i_1', i_2'-k-1)
     \end{tikzcd}
     \]
which is the generic $2$-morphism in $\Sigma([k]\times[l]^{\op})$ from $x_{\bot}$ to $x_{\top}$.
     \qedhere
\end{enumerate}
\end{proof}
 
We recall the join of simplicial sets with marking, which extends the ordinary join for simplicial sets.\footnote{The unmarked version of the join construction appears in \cite{EhlersPorterJoin}, \cite[\textsection 3]{JoyalVolumeII}, \cite[\textsection 1.2.8]{htt} and \cite[\textsection 2.4]{RiehlVerity2Cat}.  The marked version is in \cite[Obs.~34]{VerityComplicialI} or \cite[Def.~3.2.5]{EmilyNotes}.
 }

 \begin{defn}\label{DefJoin}
The \emph{join} $X\star X'$ of simplicial sets with marking is the simplicial set defined as follows.
\begin{itemize}[leftmargin=*]
    \item The set of $m$-simplices is given by
\[(X\star X')_{m}=\coprod_{k+l=m-1, k,l\geq -1}X_{k}\times X'_{l}\]
where both $X_{-1}$ and $X'_{-1}$ are singletons by definition.
\item The faces and degeneracies of a simplex $(\sigma, \sigma')\in X_k\times X'_{l}\subset (X\star X')_m$ are given by
\[
d_i(\sigma, \sigma')=\left\{\begin{array}{ll}
(d_i\sigma, \sigma') &\mbox{ if }0\leq i\leq k,\\
(\sigma, d_{i-k-1}\sigma') &\mbox{ if }k+1\leq i \leq m=k+1+l,
\end{array}
\right.
\]
and 
\[
s_i(\sigma, \sigma')=
\left\{\begin{array}{ll}
(s_i\sigma, \sigma') & \mbox{ if }0\leq i\leq k,\\
(\sigma, s_{i-k-1}\sigma')& \mbox{ if }k+1\leq i \leq m=k+1+l.
\end{array}
\right.
\]
\item A simplex $(\sigma, \sigma')$ is marked if either $\sigma$ is marked in $X$ or $\sigma'$ is marked in $X'$ (or both).
\end{itemize}
 \end{defn}

 \begin{prop}
\label{conehomotopical}
Regarding $X\star\Delta[0]$ as pointed on the $0$-simplex $x_{\top}$ coming from $\Delta[0]$, the marked join with a $0$-simplex defines a functor
 \[(-)\star \Delta[0]\colon m\sset \to m\sset_{*}\]
that is a left Quillen functor when $m\sset$ is endowed with the model structure for $(\infty,n)$-categories and $m\sset_{*}$ is endowed with the pointed model structure for $(\infty,n)$-categories.
In particular, it is homotopical.
\end{prop}

\begin{proof}
The fact that the marked join with a point $(-)\star \Delta[0]\colon m\sset \to m\sset_{*}$ defines a left adjoint functor is addressed in \cite[Def.~33]{VerityComplicialI}.
By \cref{lemmaleftquillen}, in order to prove that it is left Quillen we only need to show it respects cofibrations and it sends all types of elementary $(\infty,n)$-anodyne extensions to $(\infty,n)$-weak equivalences.
\begin{enumerate}[leftmargin=*]
    \item[(0)] The functor $(-)\star \Delta[0]$ takes cofibrations to cofibrations, as it can be seen with a routine verification using the explicit description of simplices in the suspension.
\item The functor $(-)\star \Delta[0]$ takes any complicial horn extension to an $(\infty,n)$-weak equivalence, as shown in \cite[Lemma 39]{VerityComplicialI}.
\item The functor $(-)\star \Delta[0]$ takes any complicial thinness extension to an $(\infty,n)$-weak equivalence, as shown in \cite[Lemma 39]{VerityComplicialI}.

\item The functor $(-)\star \Delta[0]$ takes each saturation extension to a saturation anodyne extension, using the isomorphism $\Delta[l]\star\Delta[0]\cong\Delta[l+1]$. Indeed, this is a consequence of \cite[Rmk 1.20]{or}, discussed in more detail in \cite[App.~D]{RiehlVerityBook}.
\item The functor $(-)\star \Delta[0]$ takes each triviality extension to an $(\infty,n)$-weak equivalence. To see this, consider a triviality anodyne extension $\Delta[m]\to\Delta[m]_t$ for $m>n$. The map $\Delta[m]\star \Delta[0]\to \Delta[m]_t\star \Delta[0]$ is then an identity on the underlying simplicial sets, with marking only differing in dimensions $m,m+1>n$. In particular, the map of simplicial set with marking can be seen as a pushout along a certain coproduct of triviality extensions $\Delta[p]\to\Delta[p]_t$ for $p>n$, and is in particular an $(\infty,n)$-weak equivalence. \qedhere
\end{enumerate}
\end{proof}

We now define the suspension\footnote{A suspension for simplicial sets (without marking) due to Kan appears in \cite{KanSemisimplicialSpectra,KanWhiteheadJoin}, and is also mentioned in  \cite[\textsection III.5]{GoerssJardine}. We refer the reader to \cite{StephanThesis} for a survey on classical simplicial suspension constructions.} of simplicial sets with marking.
 We denote by $\Delta[-1]$ the empty simplicial set.
 
 \begin{defn}
 The \emph{suspension} $\Sigma X$ of a simplicial set with marking $X$ is the simplicial set with marking defined by the pushout of simplicial sets with marking
 \[
 \begin{tikzcd}
X\star\Delta[-1]\ar[r]\ar[d]&\Delta[0]\star\Delta[-1]\ar[d]\\
X\star\Delta[0]\ar[r]&\Sigma X.
 \end{tikzcd}\]
Equivalently, $\Sigma X$ can be understood as the quotient
\[\Sigma X
\cong (X\star\Delta[0])/_X\]
of $X\star\Delta[0]$ modulo $X\star\Delta[-1]\cong X$.
In particular,
\begin{itemize}[leftmargin=*]
    \item there are two $0$-simplices, one represented by any $0$-simplex of $X$ and one represented by the $0$-simplex of $\Delta[0]$, which we call $x_{\bot}$ and $x_{\top}$ respectively.
    \item the set of $m$-simplices for $m>0$ is given by all $k$-simplices of $X$ for $0\le k\le m-1$ as well as the $m$-fold degeneracies of the two $0$-simplices $x_{\bot}$ and $x_{\top}$, namely
    \[(\Sigma X)_m\cong\{s_0^mx_{\bot}\}\amalg X_{m-1} \amalg \ldots \amalg X_0 \amalg \{s_0^mx_{\top}\}.\]
   \item the set of non-degenerate $m$-simplices for $m>0$ is given by the non-degenerate $(m-1)$-simplices of $X$.
\item  a non-degenerate $m$-simplex $\sigma$ is marked in $\Sigma X$ if and only if it is marked as an $(m-1)$-simplex of $X$.
\end{itemize} 
\end{defn}

\begin{lem}
\label{suspensionhomotopical}
Regarding $\Sigma X$ as a simplicial set with marking bipointed on $x_{\bot}$ and $x_{\top}$, the marked suspension defines a functor
\[\Sigma\colon m\sset\to m\sset_{*,*}\]
that is a left Quillen functor between the model structure for $(\infty,n)$-categories and the model structure for bipointed $(\infty,n+1)$-categories.
In particular, it is homotopical and it respects connected colimits as a functor $\Sigma\colon m\sset\to m\sset$.
\end{lem}

\begin{proof}
The fact that the suspension $\Sigma$ defines a functor is a straightforward verification, and we now describe its right adjoint functor, which we denote
$\hom\colon m\sset_{*,*}\to m\sset$, in terms of the right adjoint of $(-)\star\Delta[0]$ from \cref{conehomotopical}, which we denote $P^{\triangleright}\colon m\sset_{*}\to m\sset$.

On objects, the right adjoint is given by $(Z,a,b)\mapsto\hom_Z(a,b)$, where $\hom_Z(a,b)$ is defined by the pullback of simplicial sets with marking
 \[
 \begin{tikzcd}
\hom_Z(a,b)\ar[r]\ar[d]&P^{\triangleright}_bZ\ar[d]\\
\Delta[0]\ar[r,"a" swap]&Z,
\end{tikzcd}
\]
and the construction extends to a functor.
To see that this functor is the right adjoint to the suspension, and 
observe that a map $\Sigma X\to Z$ under $a,b$ corresponds to a commutative diagram simplicial sets with marking
\[
\begin{tikzcd}
X \arrow[rr]\arrow[d] && \Delta[0] \ar[d, "a"]\\
X\star\Delta[0]\ar[rr]&&Z\\
&\Delta[0]\arrow[ur, "b"swap]\arrow[ul] & 
\end{tikzcd}
\]
which corresponds to a commutative diagram of simplicial sets with marking
 \[
 \begin{tikzcd}
X\ar[r]\ar[d]&P_b^{\triangleright}Z\ar[d]\\
\Delta[0]\arrow[r, "a"swap]&Z,
\end{tikzcd}
\]
which corresponds to $Z\to \hom_X(a,b)$, as desired.

We now show that $\Sigma\colon m\sset\to m\sset_{*,*}$ is a left Quillen functor between the model structure for $(\infty,n+1)$-categories and the model structure for bipointed $(\infty,n+1)$-categories.
 \begin{itemize}[leftmargin=*]
     \item The functor $\Sigma$ respects cofibrations, as it can be seen with a routine verification using the explicit description of simplices in the suspension.
\item The functor $\Sigma$ respects $(\infty,n+1)$-weak equivalences. To this end, suppose that $f\colon X \to Y$ is a weak equivalence of marked simplicial sets, and consider the commutative diagram
\[
\begin{tikzcd}
\Delta[0] \ar[d, "="] & X\cong X\star \Delta[-1] \ar[l] \ar[r, hook] \ar[d,"{f\star\Delta[-1]}"] \arrow[d, shift right=1.1cm, "f", swap] & X\star \Delta[0]\ar[d,"{f\star\Delta[0]}"]\\
\Delta[0] & Y\cong Y\star \Delta[-1] \ar[l] \ar[r, hook] & Y\star \Delta[0].
\end{tikzcd}
\]
We observe that all vertical arrows are weak equivalences (the first is an identity, the second is a weak equivalence by assumption, and the third is a weak equivalence as a consequence of \cref{conehomotopical}).
Since the model structure for $(\infty,n+1)$-categories is left proper and the right horizontal arrows can be seen to be cofibrations by direct inspection, we can apply the gluing lemma (obtained combining the dual of \cite[Cor.\ 13.3.8, Prop.\ 13.3.4]{Hirschhorn}) to conclude that the map induced on the pushout diagrams 
\[
\begin{tikzcd}
\Sigma X\ar[d,"{\Sigma f}"]\\
\Sigma Y
\end{tikzcd}
\]
is an $(\infty,n+1)$-weak equivalence.
\end{itemize}

We now show that $\Sigma\colon m\sset\to m\sset_{*,*}$ is a left Quillen functor between the model structure for $(\infty,n)$-categories and the model structure for bipointed $(\infty,n+1)$-categories. Thanks to \cref{lemmaleftquillen} and previous considerations, it is enough to show that $\Sigma \Delta[n+1]  \to \Sigma \Delta[n+1]_t$ is a weak equivalence in the model structure for $(\infty,n+1)$-categories. This map is an isomorphism on the underlying simplicial sets (both isomorphic to $\Delta[n+2]/\Delta[n+1]$), and the only difference in marking is that in the right-hand side the top-dimensional $(n+2)$-simplex is marked. This means that the map $\Sigma \Delta[n+1]  \to \Sigma \Delta[n+1]_t$ is a pushout
\[\begin{tikzcd}
\Delta[n+2]\arrow[r]\arrow[d]&\Delta[n+2]_t\arrow[d]\\
\Sigma \Delta[n+1]  \arrow[r]& \Sigma \Delta[n+1]_t
\end{tikzcd}\]
of a triviality extension $\Delta[n+2]\to\Delta[n+2]_t$, and is therefore an $(\infty,n+1)$-acyclic cofibration, as desired.
 \end{proof}

 We now compare the nerve of a suspension and the suspension of a nerve.

 \begin{rmk}
 \label{comparisonnervesuspension}
 Let $\cP$ be a $1$-category. Recall from \cref{quotientoforiental} that for any $m\ge0$ there is an isomorphism of $2$-categories ${}_{\cO_2[k]}\backslash\cO_2[m+1]/_{ \cO_2[m-k]}\cong \Sigma([k]\times [m-k]^{\op})$.
 \begin{enumerate}[leftmargin=*]
     \item We have a canonical map of simplicial sets
  \[\Sigma (N\cP)\to N(\Sigma\cP),\]
 \begin{itemize}[leftmargin=*]
     \item that is identity on $0$-simplices, namely sends $x_{\bot}$ to $x_{\bot}$ and $x_{\top}$ to $x_{\top}$, and
     \item  that sends an $(m+1)$-simplex $f\colon[k] \to \cP$ with $0\leq k \leq m$ of $\Sigma(N\cP)$ for $m\ge0$ to the $(m+1)$-simplex of $N(\Sigma\cP)$ 
                \[
                \begin{tikzcd}
                \cO_2[m+1]\hphantom{\cong \Sigma([k]\times [m-k]^{\op})}\arrow[d, shift right=1.5cm, two heads] &\\[-0.3cm]
                {}_{\cO_2[k]}\backslash\cO_2[m+1]/_{ \cO_2[m-k]} \cong \Sigma([k]\times [m-k]^{\op})  \arrow[r,"{\Sigma(f\times !)}"] & \Sigma(\cP\times [0]^{\op})\cong \Sigma\cP.
                \end{tikzcd}
           \]
 \end{itemize} 
  The resulting map $\Sigma (N\cP)\to N(\Sigma\cP)$ of simplicial sets is an inclusion.
  \item The map can be enhanced to a map of simplicial sets with marking
   \[\Sigma (\NRS\cP)\to \NRS(\Sigma\cP),\]
   which is a regular inclusion.
 \end{enumerate}
 \end{rmk}

The following theorem was anticipated as \cref{ThmBIntro}, and will be proven in the next section.

 \begin{thm}
 \label{ThmB}
Let $\cP$ be a $1$-category.
\begin{enumerate}[leftmargin=*]
    \item \label{ThmBpart1}
    The canonical inclusion
  \[\Sigma (N\cP)\to N(\Sigma\cP)\]
is an inner anodyne extension, and in particular a categorical equivalence.
\item \label{ThmBpart2} The canonical inclusion
  \[\Sigma (\NRS\cP)\to \NRS(\Sigma\cP)\]
is a complicial inner anodyne extension, and in particular an $(\infty,2)$-weak equivalence.
\end{enumerate}
 \end{thm}

As applications of the theorem, we obtain the following two corollaries, which were anticipated as \cref{CorBIntro} and \cref{CorCIntro}.

\begin{cor}
\label{CorC}
Let $\mathbb I$ denote the free-living isomorphism category.
\label{completeness}
\begin{enumerate}[leftmargin=*]
    \item \label{ThmBpart1}
    The canonical map of simplicial sets \[N[1]\hookrightarrow N(\Sigma\mathbb I)\] is
    categorical equivalence.
    \item \label{ThmBpart2}
    The canonical map of simplicial sets with marking
    \[\NRS[1]\hookrightarrow \NRS(\Sigma\mathbb I)\]
    is
    an $(\infty,2)$-weak equivalence.
\end{enumerate}
\end{cor}

   \begin{proof}
   We prove Part (2); Part (1) is similar, observing that the unmarked version of \cref{suspensionhomotopical} also holds (by adapting the proof to the unmarked context using \cite[Lem.~2.1.2.3]{htt}).
We have an equivalence of (discrete) $2$-categories
\[[0]\hookrightarrow\mathbb I.\]
Since $\NRS$ is homotopical by \cref{NRSproperties},
we obtain an $(\infty,2)$-acyclic cofibration
\[\NRS[0]\hookrightarrow \NRS\mathbb I.\]
Since the suspension is homotopical by \cref{suspensionhomotopical}, we obtain an $(\infty,2)$-acyclic cofibration
\[\Sigma(\NRS[0])\hookrightarrow\Sigma(\NRS\mathbb I).\]
Since we can commute nerve and suspension up to equivalence by \cref{ThmB2}, we then obtain an $(\infty,2)$-acyclic cofibration
\[\NRS(\Sigma[0])\hookrightarrow \NRS(\Sigma\mathbb I),\]
as desired.
\end{proof}

\begin{cor}
\label{CorB}
\label{vertical}
Let $m\ge1$.
\begin{enumerate}[leftmargin=*]
    \item The canonical map of simplicial sets 
    \[
    \underbrace{N[1|1]\aamalg{N[1|0]}\dots\aamalg{N[1|0]}N[1|1]}_{m}\hookrightarrow N[1|m]
    \] 
    is
    a categorical equivalence.
    \item The canonical map of simplicial sets with marking \[ \underbrace{\NRS[1|1]\aamalg{\NRS[1|0]}\dots\aamalg{\NRS[1|0]}\NRS[1|1]}_{m}\hookrightarrow \NRS[1|m]\]
    is
    an $(\infty,2)$-weak equivalence.
\end{enumerate}
\end{cor}

\begin{proof}
We prove part (2); Part (1) is similar, observing that the unmarked version of \cref{suspensionhomotopical} also holds (by adapting the proof to the unmarked context using \cite[Lem.~2.1.2.3]{htt}).

We know by \cite[Prop.2.13]{JoyalVolumeII} that the spine inclusion
  \[
 \underbrace{\Delta[1]\aamalg{\Delta[0]}\dots\aamalg{\Delta[0]}\Delta[1]}_m\hookrightarrow \Delta[m]
 \]
 is an inner anodyne extension of simplicial sets. In fact, it can be upgraded to
 a complicial inner anodyne extension
   \[
 \underbrace{\NRS[1]\aamalg{\NRS[0]}\dots\aamalg{\NRS[0]}\NRS[1]}_m\hookrightarrow \NRS[m].
 \]
 This can be seen by either enhancing the original argument to a marked context,
 or by recognizing it as an instance of \cref{CorA}, in which $k_i=0$ for all $i$.
Since the suspension is homotopical by \cref{suspensionhomotopical}, we obtain an $(\infty,2)$-acyclic cofibration
 \[
 \underbrace{\Sigma(\NRS[1]\aamalg{\NRS[0]}\dots\aamalg{\NRS[0]}\NRS[1])}_m\hookrightarrow \Sigma\NRS[m].
 \]
 Since the suspension commutes with connected colimits by \cref{suspensionhomotopical}, we obtain an $(\infty,2)$-acyclic cofibration
\[\underbrace{\Sigma\NRS[1]\aamalg{\Sigma\NRS[0]}\dots\aamalg{\Sigma\NRS[0]}\Sigma\NRS[1]}_m\hookrightarrow \Sigma\NRS[m]\]
and using \cref{ThmB2}, we obtain an $(\infty,2)$-acyclic cofibration
 \[\underbrace{\NRS\Sigma[1]\aamalg{\NRS\Sigma[0]}\dots\aamalg{\NRS\Sigma[0]}\NRS\Sigma[1]}_m\hookrightarrow \NRS\Sigma[m],\]
 as desired.
\end{proof}

  \section{Nerve vs suspension - The proofs}
  \label{sec:SuspensionProofs}

The aim of this section is to prove \cref{ThmB}. We will prove (\ref{ThmBpart2})
and obtain (\ref{ThmBpart1}) as a corollary.

In order to do a detailed analysis of $N(\Sigma\cP)$, we will use the an explicit description of the nerve of suspension $2$-categories, that involves the following simplicial set.

  \begin{lem}[{\cite[Lemma 1.3]{NerveSuspension}}]
  Let $\cP$ be a category. 
The collection of $\cP$-matrices
\[\matn m{\cP}:=\coprod_{\substack{{k,l\ge-1,}\\{k+l=m-1}}}\big\{\sigma\colon [k]\times [l]^{\op} \to\cP\big\}\]
for $m\ge 0$ defines a simplicial set $\mat{\cP}$ in which\footnote{See \cite[Lemma 1.3]{NerveSuspension} for a precise description of the simplicial structure of $\mat{\cP}$.}
 \begin{enumerate}[leftmargin=*]
    \item faces are given by removing precisely one row or one column;
    \item degeneracies are given by doubling precisely one row or one column and inserting identities;
    \item the non-degenerate simplices are the ones where no two consecutive rows and no two consecutive columns coincide.
 \end{enumerate}
 \end{lem}
 
We have the following identification.
 
    \begin{thm}[{\cite[Theorem 1.4]{NerveSuspension}}]
 \label{simplicesasmatrices}
Let $\cP$ be a $1$-category. There is an isomorphism of simplicial sets
\[N(\Sigma{\cP})\cong\mat{\cP}.\]
In particular, an $m$-simplex of the Duskin nerve of the suspension $\Sigma{\cP}$ can be described as a functor $[k]\times [l]^{\op} \to \cP$,
together with $k,l\ge-1$ such that $k+l=m-1$.
  \end{thm}

 \begin{rmk}
 \label{rmkgrids}
 Let $\cP$ be a $1$-category.
 Under the isomorphism from \cref{simplicesasmatrices},
 \[N_m(\Sigma\cP)\cong\coprod_{\substack{{k,l\ge-1,}\\{k+l=m-1}}}\big\{\sigma\colon [k]\times [l]^{\op} \to\cP\big\},\]
 each $m$-simplex of $N\Sigma\cP$ can be uniquely described as a functor $\sigma\colon[k]\times [l]^{\op} \to \cP$, which can be pictured as a ``matrix'' valued in $\cP$
\[{
\small \begin{tikzcd}
 p_{0l} \arrow{d} \arrow{r} & p_{0 (l-1)} \arrow{d} \arrow{r} &  \cdots \arrow{r}   & p_{0 0} \arrow{d}\\
  p_{1 l} \arrow{d}\arrow{r} & p_{1 (l-1)} \arrow{d} \arrow{r} &  \cdots \arrow{r}  & p_{1 0 } \arrow{d}\\
\vdots \arrow{d}& \vdots \arrow{d}  & \ddots & \vdots \arrow{d} \\
p_{k l} \arrow{r}  & p_{k (l-1)} \arrow{r}  & \cdots   \arrow{r} & p_{k 0}.
\end{tikzcd}
}\]
In particular, for any $k$ there is a unique $k$-simplex of the form $[k]\times[-1]^{\op}\cong\varnothing\to\cP$, which can be imagined as a column of length $k$ and empty width and corresponds to the $k$-fold degeneracy of $x_{\bot}$. Similarly, for any $l\ge 0$ there is a unique $l$-simplex of the form $[-1]\times[l]^{\op}\cong\varnothing\to\cP$, which can be imagined as a row of length $l$ and empty width and corresponds to the $l$-fold degeneracy of $x_{\top}$.
 \end{rmk}

 \begin{rmk}
Let $\cP$ be a $1$-category.
 \begin{enumerate}[leftmargin=*]
     \item
     Under the identification from \cref{simplicesasmatrices},
we see that the canonical map from \cref{comparisonnervesuspension}
 \[\Sigma (N\cP)\to N(\Sigma\cP)\]
 \begin{itemize}[leftmargin=*]
    \item is the identity on $0$-simplices, namely sends $x_{\bot}$ to $x_{\bot}$ and $x_{\top}$ to $x_{\top}$, and
     \item  sends an $(m+1)$-simplex $\sigma\colon[m]\to\cP$ of $\Sigma (N\cP)$ to the $(m+1)$-simplex $\sigma\colon [m]\cong[m]\times [0]^{\op}\to\cP$.
 \end{itemize}
\item Furthermore, a non-degenerate $(m+1)$-simplex of $\Sigma(N\cP)$ is marked in $\Sigma(\NRS\cP)$
 if and only if and only if the corresponding $m$-simplex of $N\cP$ is marked in $\NRS\cP$.
 \end{enumerate}
 \end{rmk}
 
\cref{rmkgrids} suggests that the number of rows $k$
 is a relevant feature of simplices of $N(\Sigma\cP)$: the ``type''. This notion was already considered and widely discussed in \cite[\textsection2]{NerveSuspension}.

\begin{defn}
Let $\cP$ be a $1$-category.
Let $\sigma\colon[k]\times [m-k-1]^{\op}\to \cP$ be an $m$-simplex of $N(\Sigma\cP)$.
The \emph{type} of $\sigma$
is the integer $k$.
\end{defn}

\begin{rmk}
Let $\cP$ be a $1$-category.
The type $k$ of an $m$-simplex $\sigma$ of $N(\Sigma\cP)$ given in the form $\sigma\colon\cO_2[m]\to\Sigma\cP$ can also be recognized as
\[
k=
\left
\{\begin{array}{ll}
-1& \mbox{ if }\sigma=s_0^mx_{\top},\\
\max\{0\leq s\leq m\;|\;\sigma(s)=x_{\bot}\}& \mbox{ else.}\\
m& \mbox{ if }\sigma=s_0^mx_{\bot}
\end{array}
\right.
\]
This means that the $n$-functor $\sigma\colon\cO_2[m]\to \Sigma\cP$ sends the first $k+1$ objects of $\cO_2[m]$ to $x_{\bot}$ and the remaining objects to $x_{\top}$:
\[\sigma(s)=\left\{
\begin{array}{rll}
x_{\bot}&\text{for any vertex } 0\leq s \leq k\text{ of }\mathcal O_2[m]\\x_{\top}&\text{for any vertex } k+1\leq s \leq m\text{ of }\mathcal O_2[m].
\end{array}\right.\]
\end{rmk}

We will also make use of another useful feature of simplices: the ``suspect index'', and of a class of simplices of $N(\Sigma\cP)$:
the ``suspect simplices''.

\begin{defn}
\label{FirstFiltration}
Let $\cP$ be a $1$-category. Let
$\sigma\colon[k]\times [d-k]^{\op}\to \cP$
be a 
$(d+1)$-simplex  of $N(\Sigma\cP)$ of type $k$.
\begin{itemize}[leftmargin=*]
    \item The \emph{suspect index} of $\sigma$ is the minimal
$0\leq r\leq k$ such that
 for all $r\leq i\leq k$ each row $\{i\}\times [d-k]^{\op}\to \cP$ is constant. If there is no such integer, we define the suspect index to be $k+1$. 
 \item A simplex $\sigma$ is called \emph{suspect}
if it is degenerate or if it is non-degenerate of type $k$ and suspect index $r\leq k$
and  
\[\sigma\big((r-1,0)<\sigma(r,0)\big)=\id_{\sigma(r-1,0)}.\]
\end{itemize}
\end{defn}

\begin{ex}\label{FirstExSuspect}
Let $\cP$ be a $1$-category, and $g$ a non-identity morphism. Consider the following two $6$-simplices of $N(\Sigma\cP)$
\[
{
\scriptsize \begin{tikzcd}
 p_{02} \arrow{d} \arrow{r}&p_{01}\arrow{r}\arrow{d}&p_{00}\arrow{d} \\
  p_{1 2} \arrow{r} \ar[d]& p_{11}\arrow{r} \arrow[d] & p_{10}\arrow[d, "="]  \\
  p_{10} \arrow[r, "="] \ar[d, "f"] & p_{10} \arrow[r, "="] \ar[d, "f"] & p_{10} \ar[d, "f"]\\
  p_{30}\arrow[r, "="] & p_{30}\arrow[r, "="]& p_{30}
\end{tikzcd}
}
\quad\quad \quad {
\scriptsize \begin{tikzcd}
 p_{02} \arrow{d} \arrow{r}&p_{01}\arrow{r}\arrow{d}&p_{00}\arrow{d} \\
  p_{1 2} \arrow{r} \ar[d]& p_{11}\arrow{r} \arrow[d] & p_{10}\arrow[d, "g\neq \id"]  \\
  p_{20} \arrow[r, "="] \ar[d, "f"] & p_{20} \arrow[r, "="] \ar[d, "f"] & p_{20} \ar[d, "f"]\\
  p_{30}\arrow[r, "="] & p_{30}\arrow[r, "="]& p_{30}
\end{tikzcd}
}
\]
They both have type $3$, and have suspect index $2$. 
However, given that $g$ is not an identity, only the first one is a suspect simplex.
 \end{ex}

 We record for future reference the following features of the faces of a suspect simplex.
These properties, whose proof we omit, can be deduced from a careful case distinction for the types.

\begin{lem}\label{BoundarySuspect2}
Let $\cP$ be $1$-category.
Let $\sigma$ be a non-degenerate suspect $(d+1)$-simplex of $N(\Sigma\cP)$ of type $k$ and suspect index $r\leq k$.
The $a$-th face of $\sigma$
\[d_a(\sigma)\text{ is }\left\{
\begin{array}{ll}
  \text{a suspect simplex}   & \text{if $0\leq a \leq r-2$} \\
   \text{of suspect index at most $(r-1)$}   &  \text{if } a=r-1\\
    \text{of type $k-1$  and suspect index $r$}   & \text{if } a=r\\
     \text{a suspect simplex}   & \text{if $r+1\leq a \leq k$}\\
     \text{of type $k$}   & \text{if $k+1\leq a \leq k+1+l.$}
\end{array}\right.\]
\end{lem}
 
For the sake of intuition, one can verify the validity of the lemma in the example below.
 
 \begin{ex}
 The following pictures display the $a$-th face of the suspect $6$-simplex of index $2$ and type $3$ considered in \cref{FirstExSuspect}.
 \[
 \begin{array}{cccc}
{
\scriptsize \begin{tikzcd}
  p_{1 2} \arrow{r} \ar[d]& p_{11}\arrow{r} \arrow[d] & p_{10}\arrow[d, "="]  \\
  p_{10} \arrow[r, "="] \ar[d, "f"] & p_{10} \arrow[r, "="] \ar[d, "f"] & p_{10} \ar[d, "f"]\\
  p_{30}\arrow[r, "="] & p_{30}\arrow[r, "="]& p_{30}
\end{tikzcd}
}
&
{
\scriptsize \begin{tikzcd}
 p_{02} \arrow{d} \arrow{r}&p_{01}\arrow{r}\arrow{d}&p_{00}\arrow[d, "=?"] \\
  p_{10} \arrow[r, "="] \ar[d, "f"] & p_{10} \arrow[r, "="] \ar[d, "f"] & p_{10} \ar[d, "f"]\\
  p_{30}\arrow[r, "="] & p_{30}\arrow[r, "="]& p_{30}
\end{tikzcd}
}
&
{
\scriptsize \begin{tikzcd}
 p_{02} \arrow{d} \arrow{r}&p_{01}\arrow{r}\arrow{d}&p_{00}\arrow{d} \\
  p_{1 2} \arrow{r} \ar[d]& p_{11}\arrow{r} \arrow[d] & p_{10}\arrow[d, "g\neq \id"]  \\
  p_{30}\arrow[r, "="] & p_{30}\arrow[r, "="]& p_{30}
\end{tikzcd}
}
\\
a=0\leq r-2 & a=1=r-1 & a=2=r &
\end{array}
\]

\[
\begin{array}{ccc}
{
\scriptsize \begin{tikzcd}
 p_{02} \arrow{d} \arrow{r}&p_{01}\arrow{r}\arrow{d}&p_{00}\arrow{d} \\
  p_{1 2} \arrow{r} \ar[d]& p_{11}\arrow{r} \arrow[d] & p_{10}\arrow[d, "="]  \\
  p_{10} \arrow[r, "="]  & p_{10} \arrow[r, "="]  & p_{10} \\
\end{tikzcd}
}
&
{
\scriptsize \begin{tikzcd}
 p_{02} \arrow{d} \arrow{r}&p_{01}\arrow{d} \\
  p_{1 2} \arrow{r} \ar[d]& p_{11} \arrow[d]   \\
  p_{10} \arrow[r, "="] \ar[d, "f"] & p_{10}  \ar[d, "f"] \\
  p_{30}\arrow[r, "="] & p_{30}
\end{tikzcd}
}
&
\\
 r+1\leq a=3\leq k & k+1\leq a=4 \leq k+1+l &
\end{array}
\]
 \end{ex}

We now use the explicit description of simplices of $N(\Sigma\cP)$ to give an explicit description of the comparison map from \cref{comparisonnervesuspension}.

 We can now prove (\ref{ThmBpart2}) of \cref{ThmB}.

 \begin{thm}
 \label{ThmB2}
For any category $\cP$, the canonical inclusion
  \[\Sigma (\NRS\cP)\to \NRS(\Sigma\cP)\]
is a complicial inner anodyne extension, and in particular an $(\infty,2)$-weak equivalence.
 \end{thm}

In order to prove the theorem, we will add all simplices of $\NRS(\Sigma\cP)$ missing from $\Sigma(\NRS\cP)$ inductively on their (ascending) dimension $d$, their (descending) type $k$, and their (ascending) suspect index $r$.

\begin{proof}
 In order to show that the inclusion $\Sigma (\NRS\cP)\to \NRS(\Sigma\cP)$ is a complicial inner anodyne extension, we will realize it as a transfinite composite of intermediate complicial inner anodyne extensions
\[\Sigma(\NRS\cP)=:X_1\hookrightarrow X_2\hookrightarrow\dots\hookrightarrow X_{d-1}\hookrightarrow X_{d}\hookrightarrow\dots \hookrightarrow   \NRS(\Sigma\cP).\]
For $d\geq 2$, we let $X_d$ be the smallest regular subsimplicial set of $N(\Sigma\cP)$ 
containing $X_{d-1}$, all $d$-simplices of $N(\Sigma\cP)$, as well as
the suspect $(d+1)$-simplices of $N(\Sigma\cP)$.
Note that $X_1$ already contains all non-degenerate $1$-simplices of $\NRS(\Sigma\cP)$ and that there are no non-degenerate suspect $2$-simplices.
We see that the difference between $X_{d-1}$ and $X_d$ are the non-degenerate non-suspect $d$-simplices and the non-degenerate suspect $(d+1)$-simplices.

In order to show that the inclusion $X_{d-1}\hookrightarrow X_{d}$ is a complicial inner anodyne extension for all $d\geq 2$, we realize it as a composite of intermediate complicial inner anodyne extensions
\[X_{d-1}=:Y_{d} \hookrightarrow Y_{d-1} \hookrightarrow \ldots \hookrightarrow Y_{k+1}\hookrightarrow Y_{k}\hookrightarrow\ldots \hookrightarrow Y_{1}=X_{d}.\]
For $1\le k< d$,
let $Y_k$ be the
smallest regular subset of $X_{d}$ containing $Y_{k+1}$ as well as 
all non-degenerate suspect $(d+1)$-simplices $\widetilde{\tau}$ of $N(\Sigma\cP)$ of type $k$ and all non-degenerate non-suspect $d$-simplices of type $k-1$. Note that $Y_d$ already contains all non-degenerate $d$-simplices of type $d-1$ and that any suspect $(d+1)$-simplex of type $d$ is necessarily degenerate and thus can be checked to be also already in $Y_d$.
We see using \cref{BoundarySuspect2} that the difference between $Y_{k}$ and $Y_{k+1}$ are
the non-degenerate suspect $(d+1)$-simplices of type $k$ and possibly some of their faces (precisely those that are not suspect and those that are not of a higher type).

In order to show that the inclusion $Y_{k+1}\hookrightarrow Y_k$ is a complicial inner anodyne extension for $1\le k\le d-1$, we realize it as a filtration made by intermediate complicial inner anodyne extensions
\[Y_{k+1}=:W_0\hookrightarrow W_1 \hookrightarrow \ldots \hookrightarrow W_{r-1}\hookrightarrow W_r\hookrightarrow\ldots\hookrightarrow W_{k}=Y_{k}.\]
For $0<r\leq k$, we let $W_r$ be the smallest regular simplicial subset of $Y_k$
containing $W_{r-1}$ as well as 
all suspect $(d+1)$-simplices of $  \NRS(\Sigma\cP)$ of type $k$ and suspect index $r$, namely those $\widetilde{\tau}$
for which each $i$-th row constant for $r\leq i \leq k$.
Note that any simplex of suspect index $0$ is degenerate and can be checked to be already in $X_1\subset W_0$.
We see using \cref{BoundarySuspect2} that the difference between $W_{r-1}$ and $W_r$ are the non-degenerate suspect $(d+1)$-simplices $\widetilde{\tau}$ of type $k$ and suspect index $r$ and the non-degenerate non-suspect $d$-simplices $\tau$
of type $k-1$ and suspect index $r$. 

There is a bijective correspondence between the $(d+1)$- and $d$-simplices mentioned above, as follows.
On the one hand, given any such $\tau$ one can build the suspect $(d+1)$-simplex \[\widetilde{\tau}\colon [k]\times [d-k]^{\op} \to  \cP\]
of $\NRS(\Sigma\cP)$ of suspect index $r$ 
 obtained from $\tau$ by adding as $r$-row the constant map $\{r\}\times [d-k]^{\op} \to  \cP$ with value $\tau(r-1,0)$; on the other hand, given any suspect $(d+1)$-simplex $\widetilde{\tau}$, one obtains $\tau$ as $\tau=d_{r}(\widetilde{\tau})$.
 
We now record some relevant properties the $(d+1)$-simplices $\widetilde\tau$ as above.
\begin{itemize}[leftmargin=*]
    \item We argue that by induction and using \cref{BoundarySuspect2} the $r$-horn of $\widetilde{\tau}$ belongs to $W_{r-1}$; in particular, the $r$-horn defines a map of (underlying) simplicial sets
    \[\Lambda^{r}[d+1]\to W_{r-1}.\]
Indeed, using \cref{BoundarySuspect2} we see that the $a$-th face of $\widetilde{\tau}$ is already in $W_{r-1}$ for $a\neq r$ since:
\begin{itemize}[leftmargin=*]
\renewcommand\labelitemii{$\diamondsuit$}
    \item if $0\leq a \leq r-2$, the face $d_a(\widetilde\tau)$ is a suspect $d$-simplex, and in particular it belongs to $X_{d-1}\subset W_{r-1}$.
     \item if $a=r-1$, the face $d_a(\widetilde\tau)$ has suspect index at most $(r-1)$, and in particular it belongs to $W_{r-1}$ (even in $X_{d-1}$ if $r=1$).
        \item if $r+1\leq a \leq k$, the face $d_a(\widetilde\tau)$ is a suspect $d$-simplex, and in particular it belongs to $X_{d-1}\subset W_{r-1}$.
     \item if $k+1\leq a\leq k+1+l$, the face $d_a(\widetilde\tau)$ is of type $k$,
     and in particular it belongs to $Y_{k+1}\subset W_{r-1}$.
        \end{itemize}
    \item We argue that the $r$-th horn of $\widetilde{\tau}$ defines a map of simplicial sets
    \[\Lambda^r[d+1]\to W_{r-1}\]
    with  marking.
To this end, we observe that that all simplices are marked in dimensions at least $3$ in $W_{r-1}$, no non-degenerate simplices are marked in dimension $1$ in $\Lambda^r[d+1]$ (because $0<r<d+1$), and the only marked $2$-simplex of $\Lambda^r[d+1]$ is the $2$-dimensional face $\{r-1,r,r+1\}$. In particular, it is enough to show that now that this face is mapped to a degenerate $2$-simplex of $W_{r-1}$.
If $r<k$, then all the vertices of the $2$-dimensional face $\{r-1,r,r+1\}$ are mapped to $x_{\bot}$, and the $2$-simplex is mapped to the degenerate $2$-simplex at $x_{\bot}$. If $r=k$, then the $2$-dimensional face $\{r-1,r,r+1\}$ is mapped to the $2$-simplex of $N(\Sigma\cP)$
 \[{
\begin{tikzcd}
\widetilde{\tau}(r-1,0)\arrow[d, "="] \\
\widetilde{\tau}(r, 0) ,
\end{tikzcd}
}\]
which is degenerate because $\widetilde\tau$ is a suspect simplex of suspect index $r$. 
    \item If furthermore $\tau$ is marked, we argue that the $r$-th horn of $\widetilde{\tau}$ defines a map of simplicial sets with  marking
        \[\Lambda^r[d+1]'\to W_{r-1},\]
        with the simplicial set with marking $\Lambda^r[d+1]'$ defined in \cref{CompMarkAtOnce}.
        To this end, we need to show the $(r-1)$-st and $(r+1)$-st faces are mapped to a marked simplex of $W_{r-1}$. This is true when $d>2$ because all simplices in dimension at least $3$ are marked in $W_{r-1}$, and we now address the case $d=2$.
        In this case, the non-degenerate suspect $3$-simplex $\widetilde\tau$ is necessarily of the form
           \[{
\begin{tikzcd}
\widetilde{\tau}(1) \arrow[d, "\widetilde{\tau}(10)"swap] \arrow[r, "\widetilde{\tau}(10)"]&\widetilde{\tau}(0)\arrow[d, "="] \\
 \widetilde{\tau}(0) \arrow[r, "="swap] & \widetilde{\tau}(0).
\end{tikzcd}
}\]
        and in particular $k=1=r$. The zeroth face of $\widetilde\tau$ is degenerate and thus marked, and the second face of $\widetilde\tau$ must be marked because it
    is inhabited by the same $2$-morphism of $\Sigma\cP$ (so $1$-morphism of $\cP$) as $\tau$, which is marked by assumption.

\end{itemize}
 By filling all $r$-horns of suspect $(d+1)$-simplices $\widetilde{\tau}$ of $W_r$, we then obtain their $r$-th face $\tau$, which was missing in $W_{r-1}$, as well as the suspect $(d+1)$-simplex $\widetilde{\tau}$ itself.
This can be rephrased by saying that there is a pushout square
\[
\begin{tikzcd}[column sep=1.0cm]
\underset{\mathclap{\begin{subarray}{c}
  \tau \\
  \mbox{\tiny{non-marked}}
  \end{subarray}} }{\coprod} \Lambda^{r}[d+1]\amalg\underset{\mathclap{\begin{subarray}{c}
 \tau \\
  \mbox{\tiny{marked}}
  \end{subarray} }}{\coprod} \Lambda^{r}[d+1]'\arrow[d]\arrow[r] &
 \underset{\mathclap{\begin{subarray}{c}
  \tau \\
  \mbox{\tiny{non-marked}}
  \end{subarray} }}{\coprod} \Delta^{r}[d+1]\amalg \underset{\mathclap{\begin{subarray}{c}
\tau\\
  \mbox{\tiny{marked}}
  \end{subarray}} }{\coprod} \Delta^{r}[d+1]''\arrow[d]\\
W_{r-1} \arrow[r]& W_r.
\end{tikzcd}
\]
Since the involved horn inclusions are in fact inner horn inclusions,
the inclusions of simplicial sets with marking $\Lambda^r[d+1]\hookrightarrow\Delta^r[d+1]$ and $\Lambda^r[d+1]'\hookrightarrow\Delta^r[d+1]''$ are complicial inner anodyne extensions by \cref{CompMarkAtOnce}.

It follows that $W_{{r-1}}\hookrightarrow W_r$ is an anodyne for any $1\le r\le d-j$, that the inclusion $Y_{j-1}\hookrightarrow Y_j$ for any $1\le j\le d$, the inclusion $Y_{j-1}\hookrightarrow Y_j$ for any $1\le j\le d$, the inclusion $X_{d-1}\hookrightarrow X_d$ for any $d\ge1$, and $\Sigma (\NRS\cP)\to \NRS(\Sigma\cP)$ are complicial inner anodyne extensions, as desired.
\end{proof}

As an instance of \cref{underlyingcomplicialinner} (or by reading the previous proof ignoring the marking), we obtain the following corollary, which is (\ref{ThmBpart1}) of \cref{ThmBIntro}.

\begin{cor}
\label{ThmB1}
For any category $\cP$, the canonical inclusion
  \[\Sigma (N\cP)\to N(\Sigma\cP)\]
is an inner anodyne extension, and in particular a categorical equivalence.
\end{cor}

    \section{Nerve vs wedge - The results}
  \label{sec:WedgeResults}
  
  In this section, we illustrate the results and applications related to the compatibility of nerve and certain gluing construction that we call ``wedge''.
  
  \begin{defn}
  \label{wedgencat}
Let $n\in\mathbb N\cup\{\infty\}$. Let $\cA$ be an $n$-category, and $a_{\top}$ (resp. $a_{\bot}$) an object of $\cA$. The object $a_{\top}$ (resp. $a_{\bot}$) is a \emph{cosieve object} (resp.~\emph{sieve object}) if the following equivalent\footnote{The equivalence of the conditions can be seen as a special case of the argument from \cite[\textsection 2.3]{AraMaltsiniotisVers}.} conditions are met.
\begin{itemize}[leftmargin=*]
 \item Given any object $b\in\cA$, the hom $(n-1)$-category $\Map_{\cA}(a_\top,b)$ (resp.~$\Map_{\cA}(b,a_{\bot})$) 
 is given by
  \[
  \Map_{\cA}(a_{\top},b)=\left\{
 \begin{array}{lll}
 \{\id_{a_{\top}}\}& b=a_{\top}\\
 \varnothing& b\neq a_{\top}
 \end{array}
 \right.\!
 \left(\text{resp.~\vphantom{$x$}}
 \Map_{\cA}(b,a_{\bot})=\left\{
 \begin{array}{lll}
 \{\id_{a_{\bot}}\}& b=a_{\bot}\\
 \varnothing& b\neq a_{\bot}
 \end{array}
 \right.
 \!\right)
 \]

    \item The inclusion $a\colon[0]\hookrightarrow\cA$ is a cosieve (resp.~sieve), as defined in \cite[\textsection 2.3]{AraMaltsiniotisVers} under the name of \emph{cocrible} (resp.~\emph{crible}), i.e., there is an $n$-functor $\chi\colon\cA\to[1]$ that restricts to an isomorphism of $n$-categories \[\chi^{-1}\{1\}\cong\{a_{\top}\} \text{ (resp.~}\chi^{-1}\{0\}\cong\{a_{\bot}\}\text{)}.\]
\end{itemize}
\end{defn}

\begin{ex}
Let $\cP$ be a $1$-category (e.g.~$\cP=[k]$). The suspension $2$-category $\Sigma\cP$ (e.g.~$\Sigma\cP=[1|k]$) has a (unique) cosieve object, given by the last object, and a (unique) sieve object, given by the first object.
\end{ex}

 We consider the following type of pushout of $n$-categories along (co)sieve objects.
  
  \begin{defn}
  \label{DefWedge}
  \label{wedge}Let $n\in\mathbb N\cup\{\infty\}$.
  The \emph{wedge} of two $n$-categories $\cA$ endowed with a cosieve object $a_{\top}$ and $\cA'$ with a sieve object $a'_{\bot}$ is the pushout
  \[
\begin{tikzcd}
{[0]} \arrow[r,"a_{\top}"]\arrow[d,"{a'_{\bot}}" swap] & \cA\arrow[d]\\
\cA' \arrow[r] & \cA\vee\cA'.
\end{tikzcd}
\]
  \end{defn}
  
As a motivating example, the wedge construction is useful to express relation between the $n$-categories belonging to Joyal's categories $\Theta_n$ (see e.g.~\cite{JoyalDisks}).

  \begin{ex}
  \label{wedgetheta}
For any $k,k'\ge0$ (or even more generally $k,k'\in\Theta_{n-1}$), the wedge of $[1|k]$ and $[1|k']$ is isomorphic to the $2$-category belonging to $\Theta_2$ (resp.~ $n$-category belonging to $\Theta_n$) denoted
\[[1|k]\vee [1|k']\cong[2|k,k'].\]
More generally, for any $m,m'\ge0$, $k_i,k'_{i'}\ge0$ (resp.~$k_i,k'_{i'}\in\Theta_{n-1}$) for $i=1,\dots,m$ and $i'=1,\dots,m'$, the wedge of $[m|k_1,\ldots k_m]$ and $[m'|k'_{1}, \ldots, k'_{m'}]$ is isomorphic to
\[[m+m'|k_1,\ldots, k_m, k_1', \ldots, k'_{m'}].\]
  \end{ex}

  A wedge of $2$-categories maps to their product, as explained by the following.
  
  \begin{rmk}
  Let $n\in\mathbb N\cup\{\infty\}$, and $\cA$ and $\cA'$ two $n$-categories as in \cref{wedge}, in particular endowed with functors $\chi\colon\cA\to[1]$ and $\chi'\colon\cA'\to[1]$.
The inclusions
\[\cA\cong \cA \times * \xrightarrow{\id \times a'_{\bot}}\cA\times\cA'\xleftarrow{a_{\top}\times \id}*\times \cA' \cong \cA'\]
induce a canonical map
\[\cA\vee\cA'\to\cA\times\cA',\]
which fits into a commutative diagram of $n$-categories
  \[
  \begin{tikzcd}[column sep=2.5cm]
\cA \vee\cA' \arrow[r]\arrow[d] & \cA\times \cA'\arrow[d,"\chi\times \chi'"]\\
{[2]} \arrow[r,"00\to 10 \to 11", swap] & {[1]\times [1].}
\end{tikzcd}
\]
In particular, get map
\[\cA \vee\cA'\to\cA\times\cA'.\]
\end{rmk}

The maps above turns out to be an inclusion as a consequence of the following theorem. In particular,
a wedge of $n$-categories can be understood as a sub-$n$-category of the product.\footnote{The case $n=2$ of the theorem could be treated more directly with techniques from \cite[\textsection 7.2]{AraMaltsiniotisVers}, the case $n=3$ could be treated more directly with techniques from \cite[\textsection 4.3]{GagnaThesis}, and the case in which $\cA$ and $\cA'$ are suspension $2$-categories is treated in the proof of  \cite[Thm 4.4]{NerveSuspension}.}

\begin{thm}
\label{wedgesubproduct}
\label{lemmaNcat}
Let $n\in\mathbb N\cup\{\infty\}$, and $\cA$ and $\cA'$ two $n$-categories as in \cref{wedge}.
There is a pullback square of $n$-categories
\[
\begin{tikzcd}[column sep=2.5cm]
\cA \vee\cA' \arrow[r]\arrow[d] & \cA\times \cA'\arrow[d,"\chi\times \chi'"]\\
{[2]} \arrow[r,"00\to 10 \to 11"] & {[1]\times [1].}
\end{tikzcd}
\]
In particular,
    \begin{enumerate}[leftmargin=*]
    \item[(a)] the objects of $\cA\vee\cA'$ are of the form $(a,a'_{\bot})$ or $(a_{\top}, a')$ for some object $a\in \cA$ or $a'\in\cA'$.
    \item[(b)] the mapping $(n-1)$-categories are as follows
\[
\Map_{\cA\vee\cA'}((a,a'),(b,b'))\cong
\left\{
\begin{array}{lll}
\Map_{\cA}(a,b)&\mbox{ if }a'=b'=a'_{\bot}\\
\Map_{\cA'}(a',b')& \mbox{ if }a=b=a_{\top}, \\
\Map_{\cA}(a, a_{\top}) \times \Map_{\cA'}(a'_{\bot}, b')&\mbox{ if }b=a_{\top}\mbox{ and }b=a'_{\bot},\\
\varnothing& \mbox{ else.}
\end{array}
\right.
\]
\item[(c)] $\cA$ and $\cA'$ are full subcategories of $\cA\vee\cA'$.
\end{enumerate}
\end{thm}

 \begin{proof}
Let $\cQ$ be the pullback of the map $[2]\to[1]\times[1]$ along $\chi\times \chi'$
\[
\begin{tikzcd}
\cQ\arrow[r]\arrow[d] & \cA\times \cA'\arrow[d]\\
{[2]} \arrow[r] & {[1]\times [1].}
\end{tikzcd}
\]
By inspection we see that
\begin{enumerate}[leftmargin=*]
    \item[(a)] the objects of $\cQ$ are of the form $(a,a'_{\bot})$ or $(a_{\top}, a')$ for some object $a\in \cA$ or $b\in\cA'$.
    \item[(b)] the mapping $(n-1)$-categories are as follows are given by
\[
\Map_{\cQ}((a,a'),(b,b'))\cong
\left\{
\begin{array}{lll}
\Map_{\cA}(a,b)&\mbox{ if }a'=b'=a'_{\bot}\\
\Map_{\cA'}(a',b')& \mbox{ if }a=b=a_{\top}, \\
\Map_{\cA}(a, a_{\top}) \times \Map_{\cA'}(a'_{\bot}, b')&\mbox{ if }b=a_{\top}\mbox{ and }b=a'_{\bot},\\
\varnothing& \mbox{ else.}
\end{array}
\right.
\]
\item[(c)] the composition $(n-1)$-functors in the first two cases is induced by the composition in $\cA$ and $\cA'$. Moreover, the composition $(n-1)$-functors involving the third case are determined by composition in $\cA$ and in $\cA'$.
\end{enumerate}
Consider the $n$-functors
\[i_{\cA}\colon \cA \to \cQ\text{ and }i_{\cA'}\colon \cA' \to \cQ\]
defined on objects by $i_{\cA}(a)=(a,a'_{\bot})$ and $i_{\cA'}(a')=(a_{\top}, a')$, and induced by the isomorphisms above on hom-$(n-1)$-categories.
We argue that the commutative diagram of $n$-categories
\[
\begin{tikzcd}
{[0]} \arrow[r, "a'_{\bot}"]\arrow[d, "a_{\top}"swap] & \cA'\arrow[d, "i_{\cA'}"]\\
\cA\arrow[r, "i_{\cA}"swap] & \cQ
\end{tikzcd}
\]
is a pushout of $n$-categories, proving the desired statement.

In order to prove that $\cQ$ satisfies the universal property of pushouts, we suppose to be given a commutative diagram of $n$-categories formed by the solid arrows
\[
\begin{tikzcd}
{[0]} \arrow[r, "a'_{\bot}"]\arrow[d, "a_{\top}" swap] & \cA'\arrow[d, "i_{\cA'}"] \arrow[ddr, bend left=20, "\alpha'"]& \\
\cA\arrow[r, "i_{\cA}"] \arrow[drr, bend right=20, "\alpha"]& \cQ \arrow[rd, "F", dashed]& \\
&&\cD.
\end{tikzcd}
\]
We show how to construct an $n$-functor $F\colon \cQ\to \cD$ so that the diagram
commutes, and we leave the verification of the uniqueness to the reader.
 \begin{enumerate}[leftmargin=*]
\setcounter{enumi}{-1}
    \item
    We define $F$ on objects by
    \[F(a,a'_{\bot})=\alpha(a)\text{ and }F(a_{\top}, a')=\alpha'(a').\]
    \item
    We define $F$ on hom-$(n-1)$-categories
     \[F\colon\Map_{\cQ}((a,a'),(b,b'))\to\Map_{\cD}(F(a,a'),F(b,b'))\]
    \begin{itemize}[leftmargin=*]
        \item if $b=b'=a'_{\bot}$ as the functor
        \[\alpha\colon\Map_{\cA}(a,b)\xrightarrow{} \Map_{\cD}(\alpha(a), \alpha(b)).\]
       \item if $a=a'=a_{\top}$ as the functor
        \[\alpha'\colon\Map_{\cA'}(a',b')\xrightarrow{} \Map_{\cD}(\alpha'(a'), \alpha'(b')).\]
               \item if $b=a_{\top}\mbox{ and }a'=a'_{\bot}$ as the functor
        \[\alpha'(-)\circ\alpha(-)\colon\Map_{\cA}(a, a_{\top}) \times \Map_{\cA'}(a'_{\bot}, b')\xrightarrow{} \Map_{\cD}(\alpha(a), \alpha'(b')).\]
               \item otherwise as the functor
        \[\varnothing \xrightarrow{!} \Map_{\cD}(F(a,b),F(a',b')).\]
    \end{itemize}
\end{enumerate}
The fact that $F$ is compatible with identities and with most instances of composition is straightforward, and we verify compatibility with composition in one of the two interesting cases (the other one is analog). 

To this end, we need to check the commutativity of the following diagram of $(n-1)$-categories:
\[
\begin{tikzcd}[column sep=-1cm]
\Map_{\cQ}((a,a'_{\bot}),(b, a'_{\bot}))\times \Map_{\cQ}((b, a'_{\bot}),(a_{\top}, a')) \arrow[rd, "\circ_{\cQ}"]\arrow[dd, "F\times F" swap] & \\ &\Map_{\cQ}((a,a'_{\bot}),(a_{\top}, a')) \arrow[dd, "F"]\\
\Map_{\cD}(F(a,a'_{\bot}),F(b, a'_{\bot})) \times \Map_{\cD}(F(b, a'_{\bot}),F(a_{\top}, a')) \arrow[rd, "\circ_{\cD}" swap] &\\ &\Map_{\cD}(F(a,a'_{\bot}),F(a_{\top}, a')).
\end{tikzcd}
\]
Inserting the definitions and identifications above, we can identify this diagram with the following one:
\[
{\small
\begin{tikzcd}[column sep=-3.5cm]
\Map_{\cA}(a,b)\times \Map_{\cA}(b, a_{\top}) \times \Map_{\cA'}(a'_{\bot}, a') \arrow[rd, "\circ_{\cA}\times \id"]\arrow[dd, "\alpha\times \alpha\times \alpha'" swap] & \\
&\Map_{\cA}(a, a_{\top}) \times \Map_{\cA'}(a'_{\bot}, a') \arrow[dd, "\alpha\times\alpha'"]\\
 \Map_{\cD}(\alpha(a), \alpha(b))\times \Map_{\cD}(\alpha(b), \alpha(a_{\top})) \times \Map_{\cD}(\alpha'(a'_{\bot}), \alpha'(a')) \arrow[dd, "\id \times \circ_{\cD}"swap] &\\
 & \Map_{\cD}(\alpha(a), \alpha(a_{\top})) \times\Map_{\cD}(\alpha'(a'_{\bot}), \alpha'(a')) \arrow[dd, "\circ_{\cD}"]\\
\Map_{\cD}(\alpha(a),\alpha(b)) \times \Map_{\cD}(\alpha(b), \alpha'(a')) \arrow[rd, "\circ_{\cD}"swap] &\\
& \Map_{\cD}(\alpha(a), \alpha'(a')).
\end{tikzcd}
}
\]
This diagram commutes since $\alpha$ is a functor and $\circ_{\cD}$ is associative.
\end{proof}

We can define an analog wedge for simplicial sets along $0$-simplices.

 \begin{defn}
  \label{wedgeSimp}
  The \emph{wedge} of two simplicial sets with marking $X$ with a specified $0$-simplex $x_{\bot}$ and $X'$ with a specified $0$-simplex $x'$ is the pushout of simplicial sets with  marking
  \[
\begin{tikzcd}
{\Delta[0]} \arrow[r, "x"]\arrow[d, "{x'}" swap] & X\arrow[d]\\
X' \arrow[r] & X\vee X'.
\end{tikzcd}
\]
  \end{defn}

We can now compare nerve of wedge with wedge of nerve as follows.

\begin{rmk}
Let $n\in\mathbb N\cup\{\infty\}$, and $\cA$ and $\cA'$ two $n$-categories as in \cref{wedgencat}.
\begin{enumerate}[leftmargin=*]
    \item There is a commutative diagram
  \[
\begin{tikzcd}
{\Delta[0]} \arrow[r,"a_{\top}"]\arrow[d,"{a'_{\bot}}" swap] & N\cA\arrow[d]\\
N\cA' \arrow[r] & N(\cA\vee\cA').
\end{tikzcd}
\]
By the universal property of pushouts we obtain a canonical map of simplicial sets
\[N\cA\vee N\cA' \to N(\cA \vee\cA'),\]
which is an inclusion. 
Under the identification from \cref{wedgesubproduct}, this map sends an $m$-simplex of $N\cA\vee N\cA'$ of the form $\sigma\colon\cO_2[m]\to\cA$ (resp.~$\sigma'\colon\cO_2[m]\to\cA$) to the $m$-simplex of $N(\cA \vee\cA')$ given by
\[(\sigma,s_0^ma'_{\bot})\colon\cO_2[m]\to\cA\vee\cA'\text{\quad(resp.}~(s_0^ma_{\top},\sigma')\colon\cO_2[m]\to\cA\vee\cA'\text{).}\]
Moreover, a pair of $n$-functors $(\sigma,\sigma')$, where $\sigma\colon\mathcal O_2[m]\to\cA$ and $\sigma'\colon\mathcal O_2[m]\to\cA'$, defines an $m$-simplex of $N(\cA \vee\cA')$ if and only if \[\chi\sigma(s)\geq \chi'\sigma'(s)\text{ for any vertex } 0\leq s \leq m\text{ of }\mathcal O_2[m].\]
\item The map of simplicial sets can be enhanced to a map of simplicial sets with marking
\[\NRS\cA\vee \NRS\cA' \to \NRS(\cA \vee\cA'),\]
which is a regular inclusion, given that $\cA$ and $\cA'$ are full sub-$n$-categories of $\cA\vee\cA'$.
\end{enumerate}
\end{rmk}

The main result of this section is that the nerve construction commutes with the wedge construction up to a suitable notion of weak equivalence.

\begin{thm}
\label{ThmA}
Let $n\in\mathbb N\cup\{\infty\}$, and $\cA$ and $\cA'$ two $n$-categories as in \cref{wedge}.
\begin{enumerate}[leftmargin=*]
    \item
    \label{ThmApart1}
    The canonical map of simplicial sets \[N\cA\vee N\cA' \to N(\cA \vee\cA')\] is an inner anodyne extension, and in particular a categorical equivalence and a weak homotopy equivalence.
    \item
    \label{ThmApart2}
    The canonical map of simplicial sets with marking \[\NRS\cA\vee \NRS\cA' \to \NRS(\cA \vee\cA')\] is a complicial inner anodyne extension, and in particular an $(\infty,n)$-weak equivalence.
\end{enumerate}
\end{thm}

The theorem will be proven in the next section.

Recall from \cref{wedgetheta} that $2$-categories of the form $[m|k_1,\dots,k_m]$ are the objects of $\Theta_2$ (more generally, that $n$-categories of the form $[m|k_1,\dots,k_m]$ are the objects of $\Theta_n$ for $k_1,\dots,k_m\in\Theta_{n-1}$).
As an application of \cref{ThmBIntro}, we obtain the following corollary, which was anticipated as \cref{CorAIntro}.

\begin{cor}
\label{CorA}
Let $m\in\mathbb N$ and $k_1,\dots,k_m\in\mathbb{N}$ (or $k_1,\dots,k_m\in\Theta_{n-1}$).
\begin{enumerate}[leftmargin=*]
    \item The canonical map of simplicial sets \[N[1|k_1]\vee\dots\vee N[1|k_m]\hookrightarrow N[m|k_1,\dots,k_m]\] is
    an inner anodyne extension, and in particular a categorical equivalence and a weak homotopy equivalence.
    \item The canonical map of simplicial sets with marking \[\NRS[1|k_1]\vee\dots\vee\NRS[1|k_m]\hookrightarrow \NRS[m|k_1,\dots,k_m]\] is complicial inner anodyne extension, and in particular an $(\infty,n)$-weak equivalence.
\end{enumerate}
\end{cor}

  \begin{proof}
  We observe that the object $0$ of any $(m+1)$-point suspension as defined in \cite[\textsection 4]{NerveSuspension} is a sieve object, and the object $m$ of any $(m+1)$-point suspension is a cosieve object.
  Each of the two claims is proven using the corresponding statement of \cref{ThmA} by induction on $m$, specializing to $\cA=[m|k_1,\ldots , k_m]$ and $\cA'=[1|k_{m+1}]$.
  \end{proof}

  \section{Nerve vs wedge - The proofs}
  \label{sec:WedgeProofs}

The aim of this section is to prove \cref{ThmA}. We will show (\ref{ThmApart2}), and obtain (\ref{ThmApart1}) as a corollary.

 \begin{rmk}
 Let $n\in\mathbb N\cup\{\infty\}$, and $\cA$ and $\cA'$ two $n$-categories as in \cref{wedge}, in particular endowed with functors $\chi\colon\cA\to[1]$ and $\chi'\colon\cA'\to[1]$.
 \begin{enumerate}[leftmargin=*]
     \item As a consequence of \cref{wedgesubproduct},
 there is a canonical inclusion of simplicial sets
 \[N(\cA \vee\cA')\hookrightarrow N(\cA\times\cA')\cong N\cA\times N\cA'.\]
 Moreover, a pair of $n$-functors $(\rho,\rho')$, where $\rho\colon\mathcal O_2[m]\to\cA$ and $\rho'\colon\mathcal O_2[m]\to\cA'$, defines an $m$-simplex of $N(\cA \vee\cA')$ if and only if \[\chi\rho(s)\geq \chi'\rho'(s)\text{ for any vertex } 0\leq s \leq m\text{ of }\mathcal O_2[m].\]
     \item Furthermore,
a simplex $(\rho, \rho')$ of $N(\cA\vee\cA')$ is marked in $\NRS(\cA\vee\cA')$
if and only if both components $\rho$ and $\rho'$ are marked in $N\cA$ and $N\cA'$.
This means that we obtain a regular inclusion of simplicial sets with  marking
\[\NRS(\cA \vee\cA')\hookrightarrow \NRS\cA\times \NRS\cA'.\]
 \end{enumerate}
 \end{rmk}

We will make use of the following features of simplices of $N(\cA \vee\cA')$.

\begin{defn}
Let $n\in\mathbb N\cup\{\infty\}$ and $\cA$ and $\cA'$ two $n$-categories as in \cref{wedge}.
Let $(\rho,\rho')$ be an $m$-simplex of $N(\cA \vee\cA')$.
The \emph{type} of $(\rho,\rho')$
is the pair of integers $(k_{\rho}, k_{\rho'})$ defined by
\[
k_{\rho^{(\prime)}}=
\left
\{\begin{array}{ll}
-1& \mbox{ if }\chi^{(\prime)}\rho^{(\prime)}= 1,\\
\max\{0\leq s\leq m\;|\;\chi^{(\prime)}\rho^{(\prime)}(s)=0\}& \mbox{ else.}
\end{array}
\right.
\]
In particular, since $\chi\rho(s)\geq \chi'\rho'(s)$ for any vertex $0\leq s \leq m$, we have that $k_{\rho'}\geq k_{\rho}$.
\end{defn}

\begin{rmk}
The definition can be rephrased by saying that any $n$-functor $\rho\colon\cO_2[m]\to \cA\vee\cA'$ sends
\begin{itemize}[leftmargin=*]
    \item the first $k_{\rho}+1$ vertices of $\cO_2[m]$ to $\cA\setminus\{a_{\top}\}$,
    \item the next $k_{\rho'}-k_{\rho}$ vertices of $\cO_2[m]$ to $a_{\top}=a'_{\bot}\in\cA\vee\cA'$,
    \item and the final $m-k_{\rho'}$ vertices of $\cO_2[m]$ to $\cA'\setminus\{a'_{\bot}\}$.
\end{itemize} 
\end{rmk}

We will also make use of another useful feature of simplices of $N(\cA \vee\cA')$: the ``suspect index'', and of a class of simplices of $N(\cA \vee\cA')$: 
the ``suspect simplices''.
We chose the same terminology as in \cref{sec:SuspensionProofs} because these notions play similar roles as those in the argument from \cref{ThmB2}.

\begin{defn} \label{FirstFiltration}
Let $n\in\mathbb N\cup\{\infty\}$ and $\cA$ and $\cA'$ two $n$-categories as in \cref{wedge}.
Let $(\rho,\rho')$ be a $(d+1)$-simplex  of $N(\cA \vee\cA')$.
\begin{itemize}[leftmargin=*]
    \item The \emph{suspect index} of $(\rho,\rho')$ is the maximal $r$ with $k_{\rho}+1\leq r\leq k_{\rho'}$
for which there exists a simplex $\alpha'$  of $N\cA'$ such that
\[\rho'=s_{r-1}\ldots s_{k_{\rho}}\alpha',\]
and $k_{\rho}$ if such $\alpha'$ does not exist.
    \item The simplex $(\rho,\rho')$ is called \emph{suspect} if it is degenerate or in $N\cA \vee N\cA'$ or if it is of suspect index $k_{\rho}+1\leq r\leq k_{\rho'}$ and $k_{\rho'}\geq k_{\rho}+1$ and in addition 
    \[\rho=s_r\alpha\]
    for some simplex $\alpha$ of $N\cA$. 
\end{itemize}
\end{defn}

We record for future reference the faces of a suspect simplex, as well as their types.

\begin{lem}\label{BoundarySuspect3}
Let $n\in\mathbb N\cup\{\infty\}$ and $\cA$ and $\cA'$ two $n$-categories as in \cref{wedge}.
Let $(\rho,\rho')=(s_r\alpha, s_{r-1}\ldots s_{k_{\rho}}\alpha')$ be a suspect $(d+1)$-simplex of $N(\cA \vee\cA')$ of suspect index $k_{\rho}+1\leq r\leq k_{\rho'}$ which is not in $N\cA \vee N\cA'$.
The $a$-th face of $(\rho,\rho')$
\[d_a(\rho,\rho')\text{ is }\left\{
\begin{array}{ll}
  \text{a suspect simplex}   & \text{if $0\leq a \leq r-1$} \\
    \text{of type $(k_{\rho},k_{\rho'}-1)$  and suspect index $r-1$}   & \text{if } a=r\\
    \text{of type $(k_{\rho},k_{\rho'}-1)$ and suspect index $r$} & \text{if }r+1=a\leq k_{\rho'}\\
\text{of type $(k_{\rho},k_{\rho'})$} & \text{if }r+1=a=k_{\rho'}+1\\
      \text{a suspect simplex }   & \text{if $r+2\leq a \leq d+1.$}\\
\end{array}\right.\]
\end{lem}

\begin{proof}
From the simplicial identities, we obtain the formulas for the $a$-th face of $(\rho,\rho')$ is
    \[
d_a(\rho,\rho')=
    \left\{
    \begin{array}{ll}
(s_{r-1}d_a\alpha, s_{r-2}\ldots s_{k_{\rho}-1}d_a\alpha') &\mbox{if } 0\leq a \leq k_{\rho}
\\
(s_{r-1}d_a\alpha,
s_{r-2}\ldots s_{k_{\rho}}\alpha')
& \mbox{if } k_{\rho}+1\leq a <r
\\
(\alpha,
s_{r-2}\ldots s_{k_{\rho}}\alpha') 
& \mbox{if } a=r
\\
(\alpha,
s_{r-1}\ldots s_{k_{\rho}}d_{k_{\rho}+1}\alpha')
& \mbox{if } r+1=a\leq k_{\rho'}
\\
(\alpha,
s_{r-1}\ldots s_{k_{\rho}}d_{k_{\rho}+1}\alpha')
& \mbox{if } r+1=a= k_{\rho'}+1
\\
(s_rd_{a-1}\alpha, 
s_{r-1}\ldots s_{k_{\rho}}d_{a-r+k_{\rho}}\alpha')& \mbox{if } r+1<a\leq k_{\rho'}
\\
(s_rd_{a-1}\alpha, 
s_{r-1}\ldots s_{k_{\rho}}d_{a-r+k_{\rho}}\alpha')& \mbox{if } k_{\rho'}<a\leq d+1.
    \end{array}
    \right.
\]
From a careful case distinction, we obtain that the type of the $a$-th face of $(\rho,\rho')$ is 
\[
k_{d_a(\rho,\rho')}=
    \left\{
    \begin{array}{ll}
(k_{\rho}-1,k_{\rho'}-1)&\mbox{if } 0\leq a \leq k_{\rho}
\\
(k_{\rho},k_{\rho'}-1)
& \mbox{if } k_{\rho}+1\leq a <r
\\
(k_{\rho},k_{\rho'}-1)
& \mbox{if } a=r
\\
(k_{\rho},k_{\rho'}-1)
& \mbox{if } r+1=a\leq k_{\rho'}
\\
(k_{\rho},k_{\rho'})
& \mbox{if } r+1=a= k_{\rho'}+1
\\
(k_{\rho},k_{\rho'}-1)& \mbox{if } r+1<a\leq k_{\rho'}
\\
(k_{\rho},k_{\rho'})& \mbox{if } k_{\rho'}<a\leq d+1.
    \end{array}
    \right.
\]
as desired.
\end{proof}

We can now prove (\ref{ThmApart2}) of \cref{ThmAIntro}.

\begin{thm}
\label{thmA2}
Let $n\in\mathbb N\cup\{\infty\}$ and $\cA$ and $\cA'$ two $n$-categories as in \cref{wedge}. The canonical map of  simplicial sets with  marking \[\NRS\cA\vee \NRS\cA' \to \NRS(\cA \vee\cA')\] is a complicial inner anodyne extension, and in particular an $(\infty,n)$-weak equivalence.
\end{thm}

In order to prove the theorem, we will add all simplices of $\NRS(\Sigma\cP)$ missing from $\Sigma(\NRS\cP)$ inductively on their (ascending) dimension $d$, the (descending) difference of types $b:=k_{\rho'}-k_{\rho}$, the (descending) type of the second component $k_{\rho'}$, and their (descending) suspect index $r$.

\begin{proof}
 In order to show that the inclusion $N\cA\vee N\cA'\hookrightarrow N(\cA \vee\cA')$ is a complicial inner anodyne extension, we will realize it as a transfinite composite of intermediate complicial inner anodyne extensions \[N\cA\vee N\cA'=:X_0\hookrightarrow X_1\hookrightarrow \ldots\hookrightarrow X_{d-1}\hookrightarrow X_d\hookrightarrow\ldots\hookrightarrow N(\cA \vee\cA'). \]
For $d\geq 1$, we let $X_d$ be the smallest regular subsimplicial set of $N(\cA \vee\cA')$ 
containing $X_{d-1}$, all $d$-simplices of $N(\cA \vee\cA')$ 
as well as the suspect $(d+1)$-simplices of $N(\cA \vee\cA')$. 
Note that $X_0$ contains all $0$-simplices of $N(\cA \vee\cA')$ as well as all that all suspect $1$-simplices of $N(\cA \vee\cA')$ are in $X_0$ by definition.
We see using \cref{BoundarySuspect3} that the difference between $X_{d}$ and $X_{d-1}$ are the non-degenerate non-suspect $d$-simplices and the non-degenerate suspect $(d+1)$-simplices.

In order to show that the inclusion $X_{d-1}\hookrightarrow X_{d}$ is a complicial inner anodyne extension for all $d\geq 1$, we realize it as a composite of intermediate complicial inner anodyne extensions
\[X_{d-1}=:Y_{d} \hookrightarrow Y_{d-1} \hookrightarrow \ldots\hookrightarrow Y_{b+1}\hookrightarrow Y_b\hookrightarrow\ldots \hookrightarrow Y_{0}=X_{d}.\]
For $d-1\ge b\ge 0$, let $Y_b$ be the
smallest regular subset of $X_{d}$ containing $Y_{b+1}$ as well as all suspect $(d+1)$-simplices $(\widetilde\sigma,\widetilde\sigma')$ of $N(\cA \vee\cA')$
of type $(k_{\widetilde\sigma},k_{\widetilde\sigma'})$ for which $k_{\widetilde\sigma'}-k_{\widetilde\sigma}=b+1$. 
Note that any $(d+1)$-simplex of type difference $d+1$ and any $d$-simplex of type difference $d$ is in $X_0\subset Y_d$. 
The difference between $Y_{b}$ and $Y_{b+1}$ is given by
the non-degenerate suspect $(d+1)$-simplices $(\widetilde\sigma,\widetilde\sigma')$ with type difference $k_{\widetilde\sigma'}-k_{\widetilde\sigma}=b+1$ and their $d$-dimensional faces not already present in $Y_{b+1}$. These are exactly the non-degenerate, non-suspect $d$-simplices $(\sigma,\sigma')$ of $N(\cA \vee\cA')$ of type difference $k_{\sigma'}-k_{\sigma}=b$.
Indeed, on the one hand one can use  \cref{BoundarySuspect3} to check that all faces of $(\widetilde\sigma,\widetilde\sigma')$ that are not already present in $Y_{b+1}$ are non-degenerate non-suspect simplices of type difference $b$; on the other hand, any such $d$-simplex $(\sigma,\sigma')$
occurs as a face of the $(d+1)$-suspect simplex  $(s_{r}\sigma, s_{r-1}\sigma')$, with $r-1$ being the suspect index of $(\sigma,\sigma')$.
In particular, we have that $Y_0=X_d$.

In order to show that the inclusion $Y_{b+1}\hookrightarrow Y_b$ is a complicial inner anodyne extension for $d-1\ge b\ge 0$, we realize it as a filtration made by intermediate complicial anodyne extensions
\[Y_{b+1}=:Z_{d} \hookrightarrow Z_{d-1} \hookrightarrow \ldots\hookrightarrow Z_{k+1}\hookrightarrow Z_k\hookrightarrow\ldots \hookrightarrow Z_{b}=Y_b.\]
For $d> k\ge b$,
 we let $Z_k$ be
the smallest regular subset of $Y_b$ containing $Z_{k+1}$ as well as all $(d+1)$-simplices $(\widetilde\sigma,\widetilde\sigma')$ of $Y_{b}$ of type $(k_{\widetilde\sigma},k_{\widetilde\sigma'})=(k-b,k+1)$. 
Note that any $(d+1)$-simplex of type $(d-b, d+1)$ is already in $X_0\subset Z_{d}$.
The difference between $Z_{k}$ and $Z_{k+1}$ are
the non-degenerate suspect $(d+1)$-simplices of $N(\cA \vee\cA')$ of type $(k_{\widetilde\sigma},k_{\widetilde\sigma'})=(k-b,k+1)$ and their $d$-dimensional faces not already present in $Z_{k+1}$, which can be seen (using \cref{BoundarySuspect3}) to be exactly all non-degenerate, non-suspect $d$-simplices $(\sigma,\sigma')$ of $N(\cA \vee\cA')$ of type $(k-b, k)$.
In particular by definition we have that $Z_b=Y_b$.

In order to show that the inclusion $Z_{k+1}\hookrightarrow Z_k$ is a complicial inner anodyne extension for $d-1\ge k\ge b$, we realize it as a filtration made by intermediate complicial inner anodyne extensions.
\[Z_{k+1}=:W_{k+2}\hookrightarrow W_{k+1} \hookrightarrow \ldots \hookrightarrow W_{r+1}\hookrightarrow W_r\hookrightarrow\ldots\hookrightarrow W_{k-b+1}=Z_k.\]
For $k+1\ge r\ge k-b+1$, we let $W_r$ be the smallest regular simplicial subset of $Z_k$
containing $W_{r+1}$ as well as the $(d+1)$-suspect simplices 
of $Z_k$ of suspect index $r$. In particular by definition we have that $W_{k-b+1}=Z_k$.
This means that the difference between $W_{r+1}$ and $W_r$ are the non-degenerate suspect $(d+1)$-simplices $(\widetilde{\sigma}, \widetilde{\sigma}')$ of $N(\cA\vee\cA')$ of type $(k-b,k+1)$ and suspect index $r$, and their $d$-dimensional faces not already present in $W_{r+1}$, which can be seen (again using \cref{BoundarySuspect3}) to be exactly and the non-degenerate non-suspect $d$-simplices $(\sigma, \sigma')$ of type $(k-b, k)$ and suspect index $r-1$.

There is a bijective correspondence between the $(d+1)$- and $d$-simplices mentioned above, as follows. On the one hand, given any such $d$-simplex $(\sigma, \sigma')$ one can build the suspect $(d+1)$-simplex
\[(\widetilde{\sigma}, \widetilde{\sigma}'):=(s_{r}\sigma, s_{r-1}\sigma')\]
of $\NRS(\cA\vee\cA')$ of suspect index $r$ and type $(k-b, k+1)$; vice versa, given any such suspect $(d+1)$-simplex $(\widetilde{\sigma}, \widetilde{\sigma}')$, one obtains $(\sigma,\sigma')$ as $(\sigma,\sigma')=d_{r}(\widetilde{\sigma}, \widetilde{\sigma}')$.

Let $(\widetilde\sigma,\widetilde\sigma')$ be a $(d+1)$-suspect simplex in $W_r$ not in $W_{r+1}$, and let's record the following relevant properties.
\begin{itemize}[leftmargin=*]
    \item We argue that the $r$-horn of $(\widetilde\sigma,\widetilde\sigma')$ belongs to $W_{r+1}$; in particular, the $r$-horn defines a map of (underlying) simplicial sets 
    \[\Lambda^{r}[d+1]\to W_{r+1}.\]
    Indeed, using \cref{BoundarySuspect3} we see that the $a$-th face of $(\widetilde\sigma,\widetilde\sigma')$ is already in $W_{r+1}$ for $a\neq r$ since:
\begin{itemize}[leftmargin=*]
\renewcommand\labelitemii{$\diamondsuit$}
    \item if $0\leq a \leq r-1$, 
    the face
    $d_a(\widetilde\sigma,\widetilde\sigma')$ is a suspect $d$-simplex, and in particular it belongs to $X_{d-1}\subset W_{r+1}$.
     \item if $a=r+1 \leq k+1$, 
     the face $d_a(\widetilde\sigma,\widetilde\sigma')$ is a $d$-simplex of type $(k-b, k)$ and suspect index $r$, 
     and in particular it belongs to $W_{r+1}$.
     \item if $a=r+1=k+2$, 
     the face $d_a(\widetilde\sigma,\widetilde\sigma')$ is of type  $(k-b, k+1)$
     and in particular it belongs to $Y_{b+1}\subset W_{r+1}$.
    \item if $r+2\leq a \leq d+1$, the face $d_a(\widetilde\sigma,\widetilde\sigma')$ is a suspect $d$-simplex, and in particular it belongs to $X_{d-1}\subset W_{r+1}$.
        \end{itemize}
    \item We argue that the $r$-th horn of $(\widetilde\sigma,\widetilde\sigma')$ defines a map of simplicial sets
    \[\Lambda^{r}[d+1]\to W_{r+1}\]
    with  marking. To this end, we need to show that any face containing $\{r-1,r,r+1\}$ is mapped to a marked simplex of $W_{r-1}$. This is true because
    a (not necessarily top-dimensional) face of $(\widetilde\sigma,\widetilde\sigma')$ that contains the vertices $\{r-1,r,r+1\}$ is necessarily degenerate in both coordinates, given that $(\widetilde\sigma,\widetilde\sigma')=(s_r\sigma,s_{r-1}\sigma')$.
  \item If furthermore $(\sigma,\sigma')$ is marked, the $r$-th horn of $(\widetilde\sigma,\widetilde\sigma')$ defines a map of simplicial sets with  marking
        \[\Lambda^{r}[d+1]'\to W_{r+1},\]
        with the simplicial set with marking $\Lambda^{r}[d+1]'$ defined in \cref{CompMarkAtOnce}.
       To this end, we need to show that the top $r$-dimensional simplex, as well as the $(r-1)st$ and $(r+1)$-st faces are mapped to a marked simplex of $W_{r-1}$. The top-dimensional $r$-simplex, and by \cref{BoundarySuspect3} its $(r-1)$-st face, are mapped to suspect simplices of $W_{r-1}$, so in particular degenerate in both components and marked. By direct computation, or using the explicit formulas given in the proof of \cref{BoundarySuspect3}, one finds that the $(r+1)$-st face is degenerate in the second component and that the first component is the simplex $\sigma$,
       which is marked by assumption, and it is therefore mapped to a pair of marked simplices.
\end{itemize}
 We can thus fill all $r$-horns of suspect $(d+1)$-simplices of $W_r$ to obtain their $(r+1)$-th face, which was missing in $W_{r+1}$, as well as the suspect $(d+1)$-simplex itself.

In particular, the discussion shows that there is a pushout square
\[
\begin{tikzcd}[column sep=1.0cm]
\underset{\mathclap{\begin{subarray}{c}
  (\sigma,\sigma') \\
  \mbox{\tiny{non-marked}}
  \end{subarray}} }{\coprod} \Lambda^{r}[d+1]\amalg\underset{\mathclap{\begin{subarray}{c}
  (\sigma,\sigma') \\
  \mbox{\tiny{marked}}
  \end{subarray} }}{\coprod} \Lambda^{r}[d+1]'\arrow[d]\arrow[r] &
 \underset{\mathclap{\begin{subarray}{c}
  (\sigma,\sigma') \\
  \mbox{\tiny{non-marked}}
  \end{subarray} }}{\coprod} \Delta^{r}[d+1]\amalg \underset{\mathclap{\begin{subarray}{c}
  (\sigma,\sigma') \\
  \mbox{\tiny{marked}}
  \end{subarray}} }{\coprod} \Delta^{r}[d+1]''\arrow[d]\\
W_{r+1} \arrow[r]& W_r.
\end{tikzcd}
\]
The involved horn inclusions are in fact inner horn inclusions, 
so the inclusions of simplicial sets with marking $\Lambda^{r}[d+1]\hookrightarrow\Delta^{r}[d+1]$ and $\Lambda^{r}[d+1]'\hookrightarrow\Delta^{r}[d+1]''$ are complicial inner anodyne extensions by \cref{CompMarkAtOnce}.

It follows that $W_{{r+1}}\hookrightarrow W_r$ is an anodyne for any $k+1\geq r \geq k-b+1$, that the inclusion $Z_{k+1}\hookrightarrow Z_k$ for any $d-1\ge k\geq b$, the inclusion $Y_{b+1}\hookrightarrow Y_b$ for any $d-1\le b\le 0$, the inclusion $X_{d-1}\hookrightarrow X_d$ for any $d\ge1$, and $N\cA\vee N\cA'\hookrightarrow N(\cA \vee\cA')$ are complicial inner anodyne extensions, as desired.
\end{proof}


As an instance of \cref{underlyingcomplicialinner} (or by reading the previous proof ignoring the marking), we obtain the following corollary, which is (\ref{ThmApart1}) of \cref{ThmAIntro}.

\begin{cor}
\label{ThmA1}
Let $n\in\mathbb N\cup\{\infty\}$ and $\cA$ and $\cA'$ two $n$-categories as in \cref{wedge}.
The canonical map of simplicial sets \[N\cA\vee N\cA' \to N(\cA \vee\cA')\] is an inner anodyne extension, and in particular a categorical equivalence.
\end{cor}

\bibliographystyle{amsalpha}

\bibliography{ref}
\end{document}